\documentclass{classpreprint}

\usepackage{geometry}
\usepackage[utf8]{inputenc}
\usepackage[T1]{fontenc}
\usepackage[english]{babel}
\usepackage{mathtools,amsmath,amssymb,amsfonts,amsthm,mathrsfs}
\usepackage{bbm}
\mathtoolsset{showonlyrefs}
\usepackage{tikz}

\usepackage{booktabs}
\usepackage{multirow}
\usepackage{rotating,tabularx}
\usepackage{adjustbox}
\usepackage{caption}
\usepackage{enumitem}

\newtheorem{theorem}{Theorem}[section]
\newtheorem{definition}{Definition}[section]
\newtheorem{lemma}{Lemma}[section]
\newtheorem{proposition}{Proposition}[section]
\newtheorem{corollary}{Corollary}[section]
\theoremstyle{definition}
\newtheorem{example}{Example}[section]
\theoremstyle{definition}
\newtheorem{remark}{Remark}[section]

\numberwithin{equation}{section}

\begin{document}

\title[Periodic parameter-dependent $\phi$-Laplacian equations]
{Periodic solutions to parameter-dependent \\ equations with a $\phi$-Laplacian type operator}

\author[G.~Feltrin]{Guglielmo Feltrin}
\address{
Department of Mathematics, University of Torino\\
via Carlo Alberto 10, 10123 Torino, Italy}
\email{guglielmo.feltrin@unito.it}

\author[E.~Sovrano]{Elisa Sovrano}
\address{
Department of Mathematics, Computer Science and Physics, University of Udine\\
via delle Scienze 206, 33100 Udine, Italy}
\email{sovrano.elisa@spes.uniud.it}

\author[F.~Zanolin]{Fabio Zanolin}
\address{
Department of Mathematics, Computer Science and Physics, University of Udine\\
via delle Scienze 206, 33100 Udine, Italy}
\email{fabio.zanolin@uniud.it}

\thanks{Work partially supported by the Grup\-po Na\-zio\-na\-le per l'Anali\-si Ma\-te\-ma\-ti\-ca, la Pro\-ba\-bi\-li\-t\`{a} e le lo\-ro Appli\-ca\-zio\-ni (GNAMPA) of the Isti\-tu\-to Na\-zio\-na\-le di Al\-ta Ma\-te\-ma\-ti\-ca (INdAM). The first author is supported by the project ERC Advanced Grant 2013 n.~339958 ``Complex Patterns for Strongly Interacting Dynamical Systems - COMPAT''.
\\
\textbf{Preprint -- August 2018}}


\date{}

\begin{abstract}
We study the periodic boundary value problem associated with the $\phi$-Laplacian equation of the form
$(\phi(u'))'+f(u)u'+g(t,u)=s$, where $s$ is a real parameter, $f$ and $g$ are continuous functions, and
$g$ is $T$-periodic in the variable $t$.
The interest is in Ambrosetti--Prodi type alternatives which provide the existence of zero, one or two solutions depending
on the choice of the parameter $s$. We investigate this problem for a broad family of nonlinearities,
under non-uniform type conditions on $g(t,u)$ as $u\to \pm\infty$.
We generalize, in a unified framework, various classical and recent results on parameter-dependent nonlinear equations.
\end{abstract}

\maketitle

\section{Introduction}\label{section-1}
In this paper, we study the problem of existence, non-existence, and multiplicity of $T$-periodic solutions
of periodic boundary value problems associated with the $\phi$-Laplacian generalized Li\'{e}nard equation
\begin{equation*}
(\phi(u'))'+f(u)u'+g(t,u)=s,
\leqno{(\mathscr{E}_{s})}
\end{equation*}
where $s$ is a real parameter. Our aim is to unify, in a generalized setting, different classical and recent results obtained in this area concerning
the trichotomy given by zero/one/two solutions by varying the parameter $s$. Such kind of alternative can be traced back to the pioneering work by Ambrosetti and Prodi~\cite{AmPr-72}.
In more detail, we discuss several configurations of $g$ related to the works by Fabry, Mawhin, Nkashama~\cite{FMN-86}, and Bereanu, Mawhin~\cite{BeMa-08}.

Looking at \cite{FMN-86}, the Authors considered the parameter-dependent Li\'{e}nard equation
\begin{equation}\label{eq-lienard}
u''+f(u)u'+g(t,u)=s,
\end{equation}
with $g$ a continuous function, $T$-periodic in $t$, and satisfying
\begin{equation*}
\lim_{|u|\to+\infty}g(t,u)=+\infty,\quad\text{uniformly in $t$},
\end{equation*}
In this framework, they proved the existence of a value $s_{0}\in\mathbb{R}$ such that the $T$-periodic problem associated with \eqref{eq-lienard} satisfies
\begin{quote}
\textsc{Alternative by Ambrosetti--Prodi (AP)}: \textit{there exist zero, at least one or at least two solutions, provided that $s<s_{0}$, $s=s_{0}$ or $s>s_{0}$.}
\end{quote}
Actually, this kind of theorems has been extended to more general equations. In particular, the study in \cite{Ma-06} concerns
nonlinear differential operators such as the $\phi$-Laplacians and considers $(\mathscr{E}_{s})$ as well.
A typical application of the results in \cite{Ma-06} can be written for the weighted equation
\begin{equation*}
(\phi(u'))'+f(u)u'+a(t)q(u)=s+e(t),
\leqno{(\mathscr{W\!E}_{s})}
\end{equation*}
for $q$ satisfying $\lim_{|u|\to+\infty}q(u)=+\infty$, and $a,e\colon\mathbb{R}\to\mathbb{R}$ continuous $T$-periodic functions with $\min a>0$.

It is interesting to observe that phenomena similar to the one in (AP) have been discovered also for different kinds of nonlinearities $q$.
In this regard, we refer to the work by Bereanu and Mawhin in \cite{BeMa-08} for the equation
\begin{equation}\label{eq-bm}
(\phi(u'))' +q(u)=s+e(t),
\end{equation}
with $\int_{0}^{T}e=0$ and $q$ satisfying $q(u)>0$ for all $u$, along with
\begin{equation*}
\lim_{|u|\to+\infty}q(u)=0.
\end{equation*}
Indeed, in \cite{BeMa-08} the Authors, extending a previous work by Ward in \cite{Wa-02}, proved the existence of a value $s_{0}\in\mathbb{R}$ such that the $T$-periodic problem associated with \eqref{eq-bm} satisfies
\begin{quote}
\textsc{Alternative by Bereanu--Mawhin (BM)}: \textit{there exist zero, at least one or at least two solutions, provided that $s>s_{0}$, $s=s_{0}$ or $0<s<s_{0}$.}
\end{quote}
Moreover, they proved that \textit{there are no solutions also for $s<0$.} The same conclusion in (BM) was obtained in \cite{Be-08} for the $p$-Laplacian Li\'{e}nard equation
\begin{equation*}
(\phi_{p}(u'))'+f(u)u'+q(u)=s+e(t),
\end{equation*}
where $\phi_{p}(\xi) := |\xi|^{p-2}\xi$, with $p>1$.

At this point, we observe that such a kind of results suggests the fact that an Ambrosetti--Prodi type alternative of the form \textit{zero}, \textit{at least one} or \textit{at least two solutions}, may occur for a broad class of nonlinearities which reflect the behavior of parabola/bell-shaped functions.
To be more precise, we consider nonlinearities $g(t,u)$ which include, as special cases, functions of the form
\begin{equation}\label{eq-g}
g(t,u)=a(t)q(u)-e(t),
\end{equation}
with $q$ having the following behavior:
\begin{quote}
\textsc{Nonlinearity of type~$\rm{I}$.} There exist $\omega_{\pm}:=\lim_{u\to\pm\infty}q(u)$ with $q(u)<\omega_{\pm}$ for $u$ in a neighborhood of $\pm\infty$.\par\noindent
\textsc{Nonlinearity of type~$\rm{II}$.} There exist $\omega_{\pm}:=\lim_{u\to\pm\infty}q(u)$ with $q(u)>\omega_{\pm}$ for $u$ in a neighborhood of $\pm\infty$.
\end{quote}
We notice that if $\omega_{\pm}=+\infty$ then $q$ is a nonlinearity of type~$\rm{I}$, while, if $\omega_{\pm}=-\infty$ then $q$ is a nonlinearity of type~$\rm{II}$.

Moreover, when $g$ has the form as in \eqref{eq-g}, we allow the weight $a(t)$ to be non-negative but possibly vanishing on sets of positive measure, so that the uniform condition $\min a>0$ is no longer required. As a consequence, we deal with situations involving non-uniform conditions on $g(t,u)$ as $u\to\pm\infty$.

From now on, we focus our attention on the periodic boundary value problem associated with $(\mathscr{E}_{s})$
where $\phi\colon\mathbb{R}\to\phi(\mathbb{R})=\mathbb{R}$ is an increasing homeomorphism such
that $\phi(0)=0$, the map $f\colon\mathbb{R}\to\mathbb{R}$ is continuous, and the
function $g\colon\mathopen{[}0,T\mathclose{]}\times\mathbb{R}\to\mathbb{R}$
satisfies the Carath\'{e}odory conditions (cf.~\cite[p.~28]{Ha-80}).
By a \textit{solution} to $(\mathscr{E}_{s})$ we mean a function $u\colon\mathopen{[}0,T\mathclose{]}\to\mathbb{R}$ of class $\mathcal{C}^{1}$
such that $\phi(u')$ is an absolutely continuous function and the equation $(\mathscr{E}_{s})$ is satisfied for almost every $t$.
Moreover, when $u(0)=u(T)$ and $u'(0)=u'(T)$, we say that $u$ is a \textit{$T$-periodic solution}.
One could equivalently consider the function $g(t,u)$ defined for a.e.~$t\in\mathbb{R}$ and $T$-periodic in $t$. In this case, one
looks for solutions $u\colon\mathbb{R}\to\mathbb{R}$ which are $T$-periodic and satisfy $(\mathscr{E}_{s})$, as described above.

It is worth noting that equation $(\mathscr{E}_{s})$ concerns the \textit{$\phi$-Laplacian operator}
which includes all the qualitative properties of the classical \textit{$p$-Laplacian operator} $\phi_{p}$
or even some more general differential operators, such as the \textit{$(p,q)$-Laplacian operator}
defined as $\phi_{p,q}(\xi):=(|\xi|^{p-2} +|\xi|^{q-2})\xi$, with $1 < p < q$.
Such kinds of differential operators are extensively studied in the literature
for their relevance in many physical and mechanical models
(cf.~\cite{KRV-10,MaMa-98,PuSe-07}).

We present now some new results concerning the $T$-periodic BVP associated with equation $(\mathscr{W\!E}_{s})$. In this introductory summary, for sake of convenience, we assume that $a,e\in L^{\infty}(0,T)$ are such that $a(t)\geq0$ for a.e.~$t\in\mathopen{[}0,T\mathclose{]}$ with $\bar{a}:=\tfrac{1}{T}\int_{0}^{T}a(t)\,dt=1$ and $\bar{e}:=\tfrac{1}{T}\int_{0}^{T}e(t)\,dt=0$.

\begin{theorem}\label{th-0.1}
Assume that $\omega_{\pm}=+\infty$. Then, there exists $s_{0}\in \mathbb{R}$ such that:
\begin{itemize}
\item for $s<s_{0}$, equation $(\mathscr{W\!E}_{s})$ has no $T$-periodic solutions;
\item for $s =s_{0}$, equation $(\mathscr{W\!E}_{s})$ has at least one $T$-periodic solution;
\item for $s_{0}<s$, equation $(\mathscr{W\!E}_{s})$ has at least two $T$-periodic solutions.
\end{itemize}
\end{theorem}

The above theorem extends the recent result in \cite{SoZa-18} to the case of $\phi$-Laplacian operators. Actually, Theorem~\ref{th-0.1} follows from a more general result dealing with equation $(\mathscr{E}_{s})$, which extends some results in \cite{Ma-06} to locally coercive nonlinearities.

\begin{theorem}\label{th-0.2}
Assume that $\omega_{\pm}=\omega\in\mathbb{R}$ and $q(u)>\omega$ for all $|u|$ sufficiently large.
Then, there exists $s_{0}\in \mathopen{]}\omega,+\infty\mathclose{[}$ such that:
\begin{itemize}
\item for $s>s_{0}$, equation $(\mathscr{W\!E}_{s})$ has no $T$-periodic solutions;
\item for $s =s_{0}$, equation $(\mathscr{W\!E}_{s})$ has at least one $T$-periodic solution;
\item for $\omega<s<s_{0}$, equation $(\mathscr{W\!E}_{s})$ has at least two $T$-periodic solutions.
\end{itemize}
Moreover, if $q(u)>\omega$ for all $u$, then, for $s\leq\omega$, equation $(\mathscr{W\!E}_{s})$ has no $T$-periodic solutions.
\end{theorem}

The above theorem allows to consider the situation of \cite{Be-08,BeMa-08} in a nonlocal setting.

\begin{theorem}\label{th-0.3}
Assume that $\omega_{-}=+\infty$ and $q(u)\nearrow\omega_{+}\in\mathbb{R}$ for $u\to+\infty$.
Then, there exists $s_{0}\in \mathopen{]}-\infty,\omega_{+}\mathclose{[}$ such that:
\begin{itemize}
\item for $s<s_{0}$, equation $(\mathscr{W\!E}_{s})$ has no $T$-periodic solutions;
\item for $s =s_{0}$, equation $(\mathscr{W\!E}_{s})$ has at least one $T$-periodic solution;
\item for $s_{0}<s<\omega_{+}$, equation $(\mathscr{W\!E}_{s})$ has at least two $T$-periodic solutions;
\item for $s\geq\omega_{+}$, equation $(\mathscr{W\!E}_{s})$ has at least one $T$-periodic solution.
\end{itemize}
\end{theorem}

As far as we know, the above theorem covers some situations which are not treated in the literature from the point of view of the Ambrosetti--Prodi type alternatives.

We notice that Theorem~\ref{th-0.1} and Theorem~\ref{th-0.3} concern nonlinearities of type I,
instead Theorem~\ref{th-0.2} is about nonlinearities of type II. These theorems can be considered as a model to produce different related
results by means of symmetries or change of variables (see the foregoing sections).
Moreover, they are all consequences of a general theorem given in Section~\ref{section-3}.

\medskip

The plan of the paper is as follows. In Section~\ref{section-2} we introduce some preliminary results based on continuation theorems and topological degree tools developed by Man\'{a}sevich and Mawhin in \cite{MaMa-98}. Moreover, taking into account \cite{SoZa-18}, we adapt Villari's type conditions to our setting.
Section~\ref{section-3} is devoted to our main results for the parameter-dependent equation $(\mathscr{E}_{s})$. The key ingredient for the proofs is Theorem~\ref{th-deg} in Section~\ref{section-2},
combined with arguments inspired from \cite{BeMa-08,FMN-86,Ma-06}.
In the same section, following \cite{MRZ-00,Om-88}, we also recall a result of Amann, Ambrosetti and Mancini type on bounded nonlinearities (cf.~\cite{AAM-78}).
In Section~\ref{section-4}, we illustrate some applications of the main results achieved in Section~\ref{section-3} to the weighted Li\'{e}nard equation $(\mathscr{W\!E}_{s})$ and Neumann problems for radially symmetric solutions.

\section{Preliminary results}\label{section-2}

In this section we deal with the differential equation
\begin{equation}\label{eq-2.1}
(\phi(u'))'+f(u)u'+h(t,u)=0,
\end{equation}
where $\phi\colon\mathbb{R}\to\mathbb{R}$ is an increasing homeomorphism with $\phi(0)=0$, $f\colon\mathbb{R}\to\mathbb{R}$ is a continuous function and $h\colon\mathopen{[}0,T\mathclose{]}\times\mathbb{R}\to\mathbb{R}$ is a Carath\'eodory function. We denote by
$\mathcal{C}^{1}_{T}:=\{u\in\mathcal{C}^{1}(\mathopen{[}0,T\mathclose{]})\colon u(0)=u(T), \, u'(0)=u'(T)\}$, and by $\mathfrak{AC}$ the set of absolutely continuous functions. By a $T$-periodic solution to \eqref{eq-2.1} we mean a function
\begin{equation*}
u\in \mathfrak{D} := \bigl{\{}u\in \mathcal{C}^{1}_{T}\colon \phi(u')\in \mathfrak{AC}\bigr{\}},
\end{equation*}
satisfying equation \eqref{eq-2.1} for a.e.~$t$.
Our purpose is to introduce the main tools for the discussion in the subsequent section.

\subsection{Definitions and technical lemmas}\label{section-2.1}

We start by introducing two concepts: the Villari's type conditions, which are inspired by \cite{Vi-66} (cf.~also \cite{BeMa-07, MaMa-98, MaSz-16}), and the upper/lower solutions (cf.~\cite{Ma-06}).

\begin{definition}\label{def-1}
A Carath\'eodory function $h(t,u)$ satisfies the Villari's condition at $+\infty$ (at $-\infty$, respectively) if there exists $\delta=\pm1$ and $d_{0} > 0$ such that
\begin{equation}\label{eq-villari}
\delta \int_{0}^{T} h(t,u(t))\,dt > 0
\end{equation}
for each $u\in \mathcal{C}^{1}_{T}$ such that $u(t)\geq d_{0}$ ($u(t)\leq -d_{0}$, respectively) for every $t\in \mathopen{[}0,T\mathclose{]}$.
\end{definition}

\begin{definition}\label{def-2}
Let $\alpha,\beta\in \mathfrak{D}$. We say that $\alpha$ is a \textit{strict lower solution} to \eqref{eq-2.1}, if
\begin{equation}\label{eq-alpha}
(\phi(\alpha'(t)))'+ f(\alpha(t))\alpha'(t) + h(t,\alpha(t)) > 0, \quad \text{for a.e.~$t\in\mathopen{[}0,T\mathclose{]}$},
\end{equation}
and if $u$ is any $T$-periodic solution to \eqref{eq-2.1} with $u(t) \geq \alpha(t)$ for all $t\in\mathopen{[}0,T\mathclose{]}$, then $u(t) > \alpha(t)$ for all $t\in\mathopen{[}0,T\mathclose{]}$. We say that $\beta$ is a \textit{strict upper solution} to \eqref{eq-2.1}, if
\begin{equation}\label{eq-beta}
(\phi(\beta'(t)))'+ f(\beta(t))\beta'(t) + h(t,\beta(t)) < 0, \quad \text{for a.e.~$t\in\mathopen{[}0,T\mathclose{]}$},
\end{equation}
and if $u$ is any $T$-periodic solution to \eqref{eq-2.1} with $u(t) \leq \beta(t)$ for all $t\in\mathopen{[}0,T\mathclose{]}$, then $u(t) < \beta(t)$ for all $t\in\mathopen{[}0,T\mathclose{]}$.
\end{definition}

Following \cite[Chapter~3, Proposition~1.5]{DH-06},
we present now a useful criterion that guarantees when a function $\alpha$ satisfying \eqref{eq-alpha} is a strict lower solution.

\begin{lemma}\label{lem-strictlower}
Let $h\colon \mathopen{[}0,T\mathclose{]}\times \mathbb{R} \to \mathbb{R}$ be a Carath\'eodory function satisfying
\begin{itemize}
\item [$(A_{0})$] \textit{for all $t_{0}\in \mathopen{[}0,T\mathclose{]}$, $u_{0}\in \mathbb{R}$ and $\varepsilon > 0$, there exists $\delta > 0$
such that if $|t-t_{0}| < \delta$, $|u - u_{0}| < \delta$, then $|h(t,u) - h(t,u_{0})| < \varepsilon$ for a.e.~$t$.}
\end{itemize}
Let $a>0$ and $\alpha\in \mathfrak{D}$ be such that
\begin{equation*}
(\phi(\alpha'(t)))'+ f(\alpha(t))\alpha'(t) + h(t,\alpha(t))\geq a, \quad \text{for a.e.~$t\in\mathopen{[}0,T\mathclose{]}$},
\end{equation*}
then $\alpha$ is a strict lower solution to \eqref{eq-2.1}.
\end{lemma}

\begin{proof}
Let $u$ be a $T$-periodic solution to \eqref{eq-2.1} with $u(t)\geq\alpha(t)$ for every $t\in\mathopen{[}0,T\mathclose{]}$. Since $\alpha$ is not a solution to equation \eqref{eq-2.1}, there exists $t_{0}\in\mathopen{[}0,T\mathclose{]}$ such that $u(t_{0})>\alpha(t_{0})$. Suppose by contradiction that there exists a maximal interval $\mathopen{[}t_{1},t_{2}\mathclose{]}$ in $\mathopen{[}0,T\mathclose{]}$ containing $t_{0}$ such that $u(t)>\alpha(t)$ for all $t\in\mathopen{]}t_{1},t_{2}\mathclose{[}$ with $u(t_{1})=\alpha(t_{1})$ or $u(t_{2})=\alpha(t_{2})$.
Since $u(t)-\alpha(t)\geq0$ for all $t\in\mathopen{[}0,T\mathclose{]}$, then $u'(t_{1})=\alpha'(t_{1})$ or $u'(t_{2})=\alpha'(t_{2})$.
First, let us suppose that $u(t_{1})=\alpha(t_{1})$. In this manner, for a.e.~$t\in\mathopen{[}0,T\mathclose{]}$, we have
\begin{equation*}
\begin{aligned}
&(\phi(u'(t)))'-(\phi(\alpha'(t)))' \leq
\\&\leq - a-f(u(t))u'(t)- h(t,u(t))+f(\alpha(t))\alpha'(t)+ h(t,\alpha(t)).
\end{aligned}
\end{equation*}
From condition $(A_{0})$, for $\varepsilon=a/4$, there exists $\delta>0$ such that, if $|t-t_{1}| < \delta$, $|u - u(t_{1})| < \delta$, then
$|h(t,u) - h(t,u(t_{1}))| < a/4$ for a.e.~$t$. Furthermore, by continuity, let $\eta<\delta$ be such that $|u(t)-u(t_{1})|<\delta$,
$|\alpha(t)-\alpha(t_{1})|<\delta$
and $|f(u(t))u'(t)-f(\alpha(t))\alpha'(t)|<a/2$, for all $t\in\mathopen{[}t_{1},t_{1}+\eta\mathclose{]}$. Then, we have $(\phi(u'(t)))'-(\phi(\alpha'(t)))' \leq0$ and,
by an integration on $\mathopen{[}t_{1},t\mathclose{]}\subseteq\mathopen{[}t_{1},t_{1}+\eta\mathclose{]}$,
we obtain $\phi(u'(t))-\phi(\alpha'(t)) \leq 0$ for a.e.~$t\in \mathopen{[}t_{1},t_{1}+\eta\mathclose{]}$.
It follows that $u'(t)\leq\alpha'(t)$ and so $u(t)\leq\alpha(t)$, for all $t\in\mathopen{[}t_{1},t_{1}+\eta\mathclose{]}$.
Then a contradiction with the definition of the interval $\mathopen{[}t_{1},t_{2}\mathclose{]}$ occurs.
Lastly, if $u(t_{2})=\alpha(t_{2})$, then a contradiction is reached in the same way.
Finally, $\alpha$ is a strict lower solution to \eqref{eq-2.1}.
\end{proof}

A similar result holds for strict upper solutions.

\begin{lemma}\label{lem-strictupper}
Let $h\colon \mathopen{[}0,T\mathclose{]}\times \mathbb{R} \to \mathbb{R}$ be a Carath\'eodory function satisfying condition $(A_{0})$.
Let $b>0$ and $\beta\in \mathfrak{D}$ be such that
\begin{equation*}
(\phi(\beta'(t)))'+ f(\beta(t))\beta'(t) + h(t,\beta(t))\leq-b, \quad \text{for a.e.~$t\in\mathopen{[}0,T\mathclose{]}$},
\end{equation*}
then $\beta$ is a strict upper solution to \eqref{eq-2.1}.
\end{lemma}

Our approach is based on continuation theorems, hence we focus our attention on the parameter depended equation
\begin{equation}\label{eq-2.1lambda}
(\phi(u'))'+\lambda f(u)u'+ \lambda h(t,u)=0,
\end{equation}
with $\lambda\in\mathopen{]}0,1\mathclose{]}$.
In particular, the detection of some a priori bounds for solutions to \eqref{eq-2.1lambda} leads to the following.

\begin{lemma}\label{lem-2.3}
Let $h\colon \mathopen{[}0,T\mathclose{]}\times \mathbb{R} \to \mathbb{R}$ be a Carath\'eodory function. If there exists $d_{1} > 0$ such that
\begin{equation}\label{eq-h}
\int_{0}^{T} h(t,u(t))\,dt \neq 0
\end{equation}
for each $u\in \mathcal{C}^{1}_{T}$ such that $u(t)\geq d_{1}$ for every $t\in \mathopen{[}0,T\mathclose{]}$, then any $T$-periodic solution $u$ of \eqref{eq-2.1lambda} with $\lambda\in\mathopen{]}0,1\mathclose{]}$ satisfies $\min u < d_{1}$. If there exists $d_{2} > 0$ such that
\eqref{eq-h} holds for each $u\in \mathcal{C}^{1}_{T}$ such that $u(t)\leq - d_{2}$ for every $t\in \mathopen{[}0,T\mathclose{]}$, then any $T$-periodic solution $u$ of \eqref{eq-2.1lambda} with $\lambda\in\mathopen{]}0,1\mathclose{]}$ satisfies $\max u > -d_{2}$.
\end{lemma}

\begin{proof}
Let $u$ be a $T$- periodic solution to \eqref{eq-2.1lambda} with $\lambda\in\mathopen{]}0,1\mathclose{]}$. By integrating, we have
\begin{equation*}
0=\int_{0}^{T} \bigl{[}(\phi(u'(t)))'+\lambda f(u(t))u'(t)+ \lambda h(t,u(t))\bigr{]}\,dt=\lambda\int_{0}^{T}h(t,u(t))\,dt.
\end{equation*}
Suppose by contradiction that either $u(t)\geq d_{1}$ or $u(t)\leq -d_{2}$ for all $t\in\mathopen{[}0,T\mathclose{]}$,
then a contradiction follows with respect to \eqref{eq-h}.
\end{proof}

\begin{lemma}\label{lem-max-min}
Let $h\colon \mathopen{[}0,T\mathclose{]}\times \mathbb{R} \to \mathbb{R}$ be a Carath\'eodory function satisfying the following property:
\begin{itemize}
\item [$(A^{\rm{I}}_{1})$] {there exists $\gamma\in L^{1}(\mathopen{[}0,T\mathclose{]},\mathbb{R}^{+})$ such that $h(t,u)\geq-\gamma(t)$ for a.e.~$t\in\mathopen{[}0,T\mathclose{]}$ and for all $u\in\mathbb{R}$.}
\end{itemize}
Then, there exists a constant $K_{0}=K_{0}(\gamma)$ such that any $T$-periodic solution $u$ to \eqref{eq-2.1lambda} with $\lambda\in\mathopen{]}0,1\mathclose{]}$ satisfies
\begin{equation*}
\max u - \min u \leq K_{0} \quad \text{ and } \quad \|u'\|_{L^{1}}\leq K_{0}.
\end{equation*}
Moreover, for any $\ell_{1},\ell_{2}\in\mathbb{R}$ with $\ell_{1}<\ell_{2}$ there exists $K_{1}>0$ such that any $T$-periodic solution $u$ to \eqref{eq-2.1lambda} with $\lambda\in\mathopen{]}0,1\mathclose{]}$ such that $\ell_{1}\leq u(t)\leq \ell_{2}$ for all $t\in\mathopen{[}0,T\mathclose{]}$ satisfies $\|u'\|_{\infty}\leq K_{1}$.
\end{lemma}

\begin{proof}
Let $\lambda\in\mathopen{]}0,1\mathclose{]}$ and let $u$ be a $T$-periodic solution to \eqref{eq-2.1lambda}.
Let $t^{*}$ be such that $u(t^{*}) = \max u$ and define $v(t):= \max u - u(t)$, which satisfies $v'=-u'$.
By hypothesis $(A^{\rm{I}}_{1})$, we deduce that for almost every $t\in\mathopen{[}0,T\mathclose{]}$
\begin{equation}\label{disequation}
(\phi(u'(t)))' = - \lambda f(u(t))u'(t) - \lambda h(t,u(t)) \leq \lambda f(u(t))v'(t) + \gamma(t).
\end{equation}
Up to an extension of $h(\cdot, u)$ by $T$-periodicity on $\mathbb{R}$, we notice that
\begin{equation*}
\int_{t^{*}}^{t^{*}+T} f(u(\xi))v(\xi)v'(\xi) \,d\xi =0.
\end{equation*}
Multiplying \eqref{disequation} by $v(t)\geq 0$ and integrating on $\mathopen{[}t^{*},t^{*}+T\mathclose{]}$, we obtain
\begin{equation*}
\int_{t^{*}}^{t^{*}+T} (\phi(u'(\xi)))' v(\xi) \,d\xi \leq \|\gamma\|_{L^{1}} \|v\|_{\infty}.
\end{equation*}
At this point, from the properties of $\phi$ it follows straightway that for every $b>0$ there exists $K_{b}>0$ such that $\phi(\xi)\xi\geq b|\xi|-K_{b}$, for every $\xi\in\mathbb{R}$. In this manner, via an integration by parts, it follows that
\begin{align*}
&\int_{t^{*}}^{t^{*}+T} (\phi(u'(\xi)))' v(\xi) \,d\xi =\int_{t^{*}}^{t^{*}+T} \phi(u'(\xi)) u'(\xi) \,d\xi = \int_{0}^{T} \phi(u'(\xi)) u'(\xi) \,d\xi \\
&\geq \int_{0}^{T} \bigl{(}b|u'(\xi)|-K_{b}\bigr{)} \,d\xi
= b\|u'\|_{L^{1}}-K_{b}T= b\|v'\|_{L^{1}}-K_{b}T.
\end{align*}
Finally, we obtain
\begin{equation*}
b\|v'\|_{L^{1}} \leq \|\gamma\|_{L^{1}} \|v\|_{\infty} + K_{b}T \leq \|\gamma\|_{L^{1}} \|v'\|_{L^{1}} + K_{b}T = \|\gamma\|_{L^{1}} \|u'\|_{L^{1}}+ K_{b}T.
\end{equation*}
Then, taking $b>\|\gamma\|_{L^{1}}$ and $K_{0}:=K_{b}T/(b-\|\gamma\|_{L^{1}})$, we have $\|u'\|_{L^{1}}\leq K_{0}$ and
hence $\max u - \min u \leq K_{0}$.

Let $\ell_{1},\ell_{2}\in\mathbb{R}$ with $\ell_{1}<\ell_{2}$. Let $u$ be such that $\ell_{1}\leq u(t)\leq \ell_{2}$
for all $t\in\mathopen{[}0,T\mathclose{]}$. By the Carath\'{e}odory condition on $h$ and the boundedness of $u$,
there exists a constant $c>0$ (independent of $u$ and $\lambda\in\mathopen{]}0,1\mathclose{]}$) such that
$u'(t)\in\phi^{-1}(\mathopen{[}-c,c\mathclose{]})$, and so there exists $K_{1}>0$ such that $\|u'\|_{\infty}\leq K_{1}$.
\end{proof}

\begin{lemma}\label{lem-max-min-2}
Let $h\colon \mathopen{[}0,T\mathclose{]}\times \mathbb{R} \to \mathbb{R}$ be a Carath\'eodory function satisfying the following property:
\begin{itemize}
\item[$(A^{\rm{II}}_{1})$] there exists $\gamma\in L^{1}(\mathopen{[}0,T\mathclose{]},\mathbb{R}^{+})$ such that $h(t,u)\leq\gamma(t)$ for a.e.~$t\in\mathopen{[}0,T\mathclose{]}$ and for all $u\in\mathbb{R}$.
\end{itemize}
Then, there exists a constant $K_{0}=K_{0}(\gamma)$ such that any $T$-periodic solution $u$ to \eqref{eq-2.1lambda} with $\lambda\in\mathopen{]}0,1\mathclose{]}$ satisfies
\begin{equation*}
\max u - \min u \leq K_{0} \quad \text{ and } \quad \|u'\|_{L^{1}}\leq K_{0}.
\end{equation*}
Moreover, for any $\ell_{1},\ell_{2}\in\mathbb{R}$ with $\ell_{1}<\ell_{2}$ there exists $K_{1}>0$ such that any $T$-periodic solution $u$ to \eqref{eq-2.1lambda} with $\lambda\in\mathopen{]}0,1\mathclose{]}$ such that $\ell_{1}\leq u(t)\leq \ell_{2}$ for all $t\in\mathopen{[}0,T\mathclose{]}$ satisfies $\|u'\|_{\infty}\leq K_{1}$.
\end{lemma}

\begin{proof}
Let $\lambda\in\mathopen{]}0,1\mathclose{]}$ and let $u$ be a $T$-periodic solution to \eqref{eq-2.1lambda}.
We enter in the same setting of Lemma~\ref{lem-max-min} via the change of variable
$x:=-u$ which leads to the study of
\begin{equation*}
(\tilde{\phi}(x'))'+\lambda \tilde{f}(x)x'+\lambda \tilde{h}(t,x)=0,
\end{equation*}
where $\tilde{\phi}(\xi) = -\phi(-\xi)$, $\tilde{f}(\xi) = f(-\xi)$, and $ \tilde{h}(t,\xi) = - h(t,-\xi)$.
\end{proof}

\subsection{Continuation theorem and abstract results}\label{section-2.2}

We introduce the fixed point operator and the continuation theorem for the more general periodic boundary value problem
\begin{equation}\label{eq-phiF}
(\phi(u'))' + F(t,u,u') = 0, \quad u(0)=u(T), \quad u'(0)=u'(T),
\end{equation}
where $F\colon\mathopen{[}0,T\mathclose{]}\times\mathbb{R}\times\mathbb{R}\to\mathbb{R}$ is a Carath\'eodory function.
We consider the following Banach spaces $X := \mathcal{C}^{1}_{T}$, endowed with the norm $\|u\|_{X}:=\|u\|_{\infty}+\|u'\|_{\infty}$,
and $Z:=L^{1}(0,T)$, with the standard norm $\|\cdot\|_{L^{1}}$.
In the same spirit of \cite{MaMa-98}, we define the continuous projectors
$P\colon X\to X$ by $P u:=u(0)$, and
$Q\colon Z\to Z$ by $Q u:=\tfrac{1}{T}\int_{0}^{T}u(t)\,dt$.
In the sequel, we also denote by $Q$ the mean value operator defined on subspaces of $Z$.
We introduce the following Nemytskii operator
$N\colon X\to Z$ by $(Nu)(t):= -F(t,u(t),u'(t))$ for $t\in\mathopen{[}0,T\mathclose{]}$.

At this point, following \cite{MaMa-98}, one has that $u$ is a solution of problem \eqref{eq-phiF} if and only if $u$ is a fixed point of the completely continuous operator $\mathcal{G}\colon X\to X$ defined as
\begin{equation*}
\mathcal{G}u:=Pu+QNu+\mathcal{K}Nu,\quad u\in X,
\end{equation*}
where $\mathcal{K}\colon Z\to X$ is the map which, to any $w\in Z$, associates the unique $T$-periodic solution $u(t)$ of the problem
\begin{equation*}
(\phi(u'))'=w(t)-\dfrac{1}{T}\int_{0}^{T}w(t)\,dt, \quad u(0)=0.
\end{equation*}

Let us consider the periodic parameter-dependent problem
\begin{equation}\label{pb-lambda}
(\phi(u'))' + \lambda F(t,u,u') = 0, \quad u(0)=u(T), \quad u'(0)=u'(T), \quad \lambda\in\mathopen{]}0,1\mathclose{]}.
\end{equation}
We are now ready to state the following continuation theorem, adapted from \cite{MaMa-98}, where by $\mathrm{d}_{\mathrm{LS}}(Id-\mathcal{G},\Omega,0)$ we denote the \textit{Leray--Schauder degree} of $Id-\mathcal{G}$ in $\Omega$, with $\Omega\in X$ an open bounded set, and by $\mathrm{d}_{\mathrm{B}}$ we indicate the finite-dimensional \textit{Brouwer degree}.
We refer to \cite[Theorem~3.1]{MaMa-98} for the proof of the following theorem (see also \cite[Section~3]{FeZa-17}).

\begin{theorem}\label{MaMa-th}
Let $\Omega$ be an open bounded set in $X$ such that the following conditions hold:
\begin{itemize}
\item for each $\lambda\in\mathopen{]}0,1\mathclose{]}$ problem \eqref{pb-lambda} has no solution on $\partial\Omega$;
\item the equation
$F^{\#}(\xi):=\tfrac{1}{T}\int_{0}^{T}F(t,\xi,0)\,dt = 0$
has no solution on $\partial\Omega\cap\mathbb{R}$.
\end{itemize}
Then, $\mathrm{d}_{\mathrm{LS}}(Id-\mathcal{G},\Omega,0)=\mathrm{d}_{\mathrm{B}}(F^{\#},\Omega\cap\mathbb{R},0)$.
Moreover, if the Brouwer degree $\mathrm{d}_{\mathrm{B}}(F^{\#},\Omega\cap\mathbb{R},0)\neq0$, then problem \eqref{eq-2.1} has a solution in ${\Omega}$.
\end{theorem}

Dealing with equation \eqref{eq-2.1}, we consider now the special form of the Carath\'eodory function
$F(t,u,u')=f(u)u'+h(t,u)$
and the following result holds.

\begin{theorem}\label{th-deg}
Let $f\colon\mathbb{R}\to\mathbb{R}$ be a continuous function. Let $h\colon \mathopen{[}0,T\mathclose{]}\times \mathbb{R} \to \mathbb{R}$
be a Carath\'eodory function satisfying $(A^{\rm{I}}_{1})$ and the Villari's condition at $-\infty$ with $\delta=1$.
Suppose there exists $\beta\in \mathfrak{D}$ which is a strict upper solution for equation \eqref{eq-2.1}.
Then, \eqref{eq-2.1} has at least a $T$-periodic solution $\tilde{u}$ such that $\tilde{u} <\beta$.
Moreover, there exist $R_{0} \geq d_{0}$ and $K_{1} > 0$, such that for each $R > R_{0}$ and
$K > K_{1}$,
we have
\begin{equation*}
\mathrm{d}_{\mathrm{LS}}(Id-\mathcal{G},\Omega,0) = -1
\end{equation*}
for $\Omega = \Omega^{\rm{I}}(R,\beta,K):=\{u \in \mathcal{C}^{1}_{T} \colon - R < u(t) < \beta(t),\; \forall\, t\in \mathopen{[}0,T\mathclose{]}, \; \|u'\|_{\infty} < K\}$.
\end{theorem}

\begin{proof}
First of all, we introduce the truncated function
\begin{equation*}
\hat{h}(t,u) :=
\begin{cases}
\, h(t,u), & \text{if $u\leq\beta(t)$,} \\
\, h(t,\beta(t)), & \text{if $u\geq\beta(t)$,}
\end{cases}
\end{equation*}
and consider the parameter-dependent equation
\begin{equation}\label{eq-trunc}
(\phi(u'))' + \lambda f(u)u' + \lambda \hat{h}(t,u) = 0, \quad \lambda\in\mathopen{]}0,1\mathclose{]}.
\end{equation}
By the assumptions on $h$, it is easy to prove that $\hat{h}$ satisfies condition $(A^{\rm{I}}_{1})$.
Then, we can apply Lemma~\ref{lem-max-min} and obtain that any $T$-periodic solution $u$ to \eqref{eq-trunc} with $\lambda\in\mathopen{]}0,1\mathclose{]}$ satisfies
$\max u - \min u \leq K_{0}$ and $\|u'\|_{L^{1}}\leq K_{0}$, for some constant $K_{0}$.
Let $d_{2}>\max\{d_{0},\|\beta\|_{\infty}\}$. By Lemma~\ref{lem-2.3} we deduce that $\max u > - d_{2}$ and so $\min u > -K_{0}-d_{2}=: -R_{0}$.

We claim that, for any $T$-periodic solution $u$ to \eqref{eq-trunc} with $\lambda\in\mathopen{]}0,1\mathclose{]}$, there exists $\hat{t}\in\mathopen{[}0,T\mathclose{]}$ such that $u(\hat{t})<\beta(\hat{t})$. Indeed, if by contradiction $u(t)\geq\beta(t)$ for all $t\in\mathopen{[}0,T\mathclose{]}$, then $u$ is a $T$-periodic solution to
\begin{equation*}
(\phi(u'))' + \lambda f(u)u' + \lambda h(t,\beta(t)) = 0, \quad \lambda\in\mathopen{]}0,1\mathclose{]}.
\end{equation*}
By an integration, we have $\int_{0}^{T}h(t,\beta(t))\,dt=0$.
The strict upper solution $\beta$ is $T$-periodic and satisfies \eqref{eq-beta}, then we obtain $\int_{0}^{T}h(t,\beta(t))\,dt<0$, a contradiction.
Therefore $\min u < \|\beta\|_{\infty}$ and so $\max u < \|\beta\|_{\infty}+ K_{0} < R_{0}$.

An application of Lemma~\ref{lem-max-min} in the framework of \eqref{eq-trunc} (with $\ell_{1}=-R_{0}$, $\ell_{2}:=\|\beta\|_{\infty}+ K_{0}$) guarantees the
existence of a constant $\hat{K}_{0}$ such that $\|u'\|_{\infty} \leq \hat{K}_{0}$,
for any $T$-periodic solution $u$ to \eqref{eq-trunc} with $\lambda\in\mathopen{]}0,1\mathclose{]}$.

We deduce that the Leray--Schauder degree $\mathrm{d}_{\mathrm{LS}}(Id-\hat{\mathcal{G}},\Gamma,0)$ is well-defined on any open and bounded set
\begin{equation*}
\Gamma := \bigl{\{} u \in \mathcal{C}^{1}_{T} \colon - R < u(t) < C,\; \forall\, t\in \mathopen{[}0,T\mathclose{]}, \; \|u'\|_{\infty} < \hat{K}\}
\end{equation*}
for any $R>R_{0}$, $C\geq\|\beta\|_{\infty}+ K_{0}$ and $\hat{K}>\hat{K}_{0}$.

Now we introduce the average scalar map
$\hat{h}^{\#} \colon \mathbb{R}\to\mathbb{R}$, defined by $\hat{h}^{\#}(\xi) := \tfrac{1}{T}\int_{0}^{T}\hat{h}(t,\xi) \,dt$, for $\xi\in\mathbb{R}$.
We notice that $\hat{h}^{\#}(-R)>0$, by the Villari's condition at $-\infty$, and $\hat{h}^{\#}(c)<0$, taking $c\geq\|\beta\|_{\infty}$.
As a consequence of Theorem~\ref{MaMa-th}, we have
\begin{equation*}
\mathrm{d}_{\mathrm{LS}}(Id-\hat{\mathcal{G}},\Gamma,0) = \mathrm{d}_{\mathrm{B}}(\hat{F}^{\#},\Gamma\cap\mathbb{R},0) = -1,
\end{equation*}
and so problem \eqref{eq-trunc} with $\lambda=1$ has at least a solution $\tilde{u}$ in $\Gamma$, more precisely $\tilde{u}$ satisfies $-R < \tilde{u}(t) < C$, for all $t\in \mathopen{[}0,T\mathclose{]}$, and
$\|\tilde{u}'\|_{\infty} < \hat{K}$.

We claim that $\tilde{u}(t)\leq\beta(t)$ for all $t\in \mathopen{[}0,T\mathclose{]}$.
We have already proved that there exists $t_{*}\in \mathopen{[}0,T\mathclose{]}$ such that $\tilde{u}(t_{*})<\beta(t_{*})$. Suppose by contradiction that there exists $t^{*}\in \mathopen{[}0,T\mathclose{]}$ such that $\tilde{u}(t^{*})>\beta(t^{*})$. Let $\mathopen{]}t_{1},t_{2}\mathclose{[}$ be the maximal open interval containing $t^{*}$ such that $\tilde{u}>\beta$. Then $\tilde{u}(t_{1})=\beta(t_{1})$ and $\tilde{u}(t_{2})=\beta(t_{2})$. Moreover $\tilde{u}'(t_{1})\geq\beta'(t_{1})$ and $\tilde{u}'(t_{2})\leq\beta'(t_{2})$, and so $\phi(\tilde{u}'(t_{1}))\geq\phi(\beta'(t_{1}))$ and $\phi(\tilde{u}'(t_{2}))\leq\phi(\beta'(t_{2}))$, due to the monotonicity of the homeomorphism $\phi$.
Next, by an integration and recalling the definition of $\hat{h}$, we have
\begin{equation*}
\begin{aligned}
0 &\geq \int_{t_{1}}^{t_{2}} \bigl{[}(\phi(\tilde{u}'(t)))'-(\phi(\beta'(t)))'\bigr{]} \,dt \\
&>\int_{t_{1}}^{t_{2}} \bigl{[}f(\tilde{u}(t))\tilde{u}'(t)-f(\beta(t))\beta'(t)\bigr{]} \,dt+ \int_{t_{1}}^{t_{2}}
\bigl{[}\hat{h}(t,\tilde{u}(t))-\hat{h}(t,\beta(t))\bigr{]} \,dt = 0
\end{aligned}
\end{equation*}
and a contradiction is found.
Then $\tilde{u}(t)\leq\beta(t)$ for all $t\in \mathopen{[}0,T\mathclose{]}$. Hence, $\tilde{u}$ is a solution of \eqref{eq-2.1}
and, since $\beta$ is a strict upper solution, $\tilde{u}(t)<\beta(t)$ for all $t\in \mathopen{[}0,T\mathclose{]}$.

As a final step, we apply Lemma~\ref{lem-max-min} in the framework of \eqref{eq-2.1} and we obtain a constant $K_{1}>0$ such that $\|u'\|_{\infty} \leq K_{1}$, for any $T$-periodic solution $u$ to \eqref{eq-2.1}. We reach the thesis via the excision property of the topological degree.
\end{proof}

Analogously we obtain the following result.

\begin{theorem}\label{th-deg-dual}
Let $f\colon\mathbb{R}\to\mathbb{R}$ be a continuous function. Let $h\colon \mathopen{[}0,T\mathclose{]}\times \mathbb{R} \to \mathbb{R}$
be a Carath\'eodory function satisfying $(A^{\rm{II}}_{1})$ and the Villari's condition at $+\infty$ with $\delta=-1$.
Suppose there exists $\alpha\in \mathfrak{D}$ which is a strict lower solution for equation \eqref{eq-2.1}.
Then, \eqref{eq-2.1} has at least a $T$-periodic solution $\tilde{u}$ such that $\tilde{u} >\alpha$.
Moreover, there exist $R_{0} \geq d_{0}$ and $K_{1} > 0$, such that for each $R > R_{0}$ and
$K > K_{1}$,
we have
\begin{equation*}
\mathrm{d}_{\mathrm{LS}}(Id-\mathcal{G},\Omega,0) = -1
\end{equation*}
for $\Omega = \Omega^{\rm{II}}(R,\alpha,K):=\{u \in \mathcal{C}^{1}_{T} \colon \alpha(t) < u(t) < R,\; \forall\, t\in \mathopen{[}0,T\mathclose{]}, \; \|u'\|_{\infty} < K\}$.
\end{theorem}

\begin{remark}\label{rem-2.1}
Assuming $(A_{0})$ and
given $\beta\in \mathfrak{D}$ an upper solution to \eqref{eq-2.1} (or in other words satisfying the weaker form of \eqref{eq-beta}),
one can still prove the existence of a $T$-periodic solution $\tilde{u}$ with $\tilde{u}\leq \beta$ under the weaker inequality in \eqref{eq-villari},
namely $\int_{0}^{T} h(t,u(t))\,dt \geq 0$ for each $u \leq -c_0$. To this purpose, we introduce the
auxiliary function
\begin{equation*}
h^{\varepsilon}(t,u):=h(t,u)+\varepsilon \min\{1,\max\{-1, - u -\|\beta\|_{\infty} -1\}\}, \quad\text{for }\;
\varepsilon>0.
\end{equation*}
Now $\beta$ becomes a strict upper solution for the modified equation $(\phi(u'))' + f(u)u' + h^{\varepsilon}(t,u) = 0$
(for each $\varepsilon > 0$). Moreover, the Villari's condition holds in the original strict form (for $u \leq - d_0$ with
$d_0 > \max\{c_0,\|\beta\|_{\infty} +1\}$). It is easy to check that Theorem~\ref{th-deg} can be applied to obtain the
existence of a $T$-periodic solution $\tilde{u}^{\varepsilon}$ to $(\phi(u'))' + f(u)u' + h^{\varepsilon}(t,u) = 0$
such that $-R_{0}<\tilde{u}^{\varepsilon}\leq\beta$. The constants
$R_{0}$ and $K_1$ can be chosen uniformly with respect to $\varepsilon$ due to particular form of the (bounded) perturbation.
An application of Ascoli--Arzel\`a theorem leads to the existence of a solution $\tilde{u}$ for
\eqref{eq-2.1}. An analogous weaker formulation of Theorem~\ref{th-deg-dual} holds too.
$\hfill\lhd$
\end{remark}

\section{Main results}\label{section-3}

In this section we present our main results concerning $T$-periodic solutions to the parameter-dependent equation
\begin{equation*}
(\phi(u'))'+f(u)u'+g(t,u)=s.
\leqno{(\mathscr{E}_{s})}
\end{equation*}
Along the section, we assume that $\phi\colon\mathbb{R}\to\mathbb{R}$ is an increasing homeomorphism with $\phi(0)=0$, the map $f\colon\mathbb{R}\to\mathbb{R}$ is continuous, the function $g\colon\mathopen{[}0,T\mathclose{]}\times\mathbb{R}\to\mathbb{R}$ satisfies Carath\'{e}odory conditions, and $s\in\mathbb{R}$. Furthermore, we introduce the following condition.
\begin{itemize}
\item [$(H_{0})$] \textit{For all $t_{0}\in \mathopen{[}0,T\mathclose{]}$, $u_{0}\in \mathbb{R}$ and $\varepsilon > 0$, there exists $\delta > 0$ such that if $|t-t_{0}| < \delta$, $|u - u_{0}| < \delta$, then $|g(t,u) - g(t,u_{0})| < \varepsilon$  for a.e.~$t$.}
\end{itemize}
In the first result, the following hypotheses will be considered as well.
\begin{itemize}
\item[$(H^{\rm{I}}_{1})$]
\textit{There exists $\gamma_{0}\in L^{1}(\mathopen{[}0,T\mathclose{]},\mathbb{R}^{+})$ such that $g(t,u) \geq - \gamma_{0}(t)$, for all $u\in\mathbb{R}$ and a.e.~$t\in\mathopen{[}0,T\mathclose{]}$.}
\item[$(H^{\rm{I}}_{2})$]
\textit{There exist $u_{0},\,g_{0}\in \mathbb{R}$ such that $g(t,u_{0})\leq g_{0}$ for a.e.~$t\in \mathopen{[}0,T\mathclose{]}$.}
\item[$(H^{\rm{I}}_{3})$]
\textit{There exist $\sigma>\max\{0,g_{0}\}$ and $d > 0$ such that $\tfrac{1}{T}\int_{0}^{T} g(t,u(t))\,dt > \sigma$
for each $u\in \mathcal{C}^{1}_{T}$ such that $u(t) \leq - d$ for all $t\in \mathopen{[}0,T\mathclose{]}$.}
\item[$(H^{\rm{I}}_{4})$]
\textit{There exist $\sigma>\max\{0,g_{0}\}$ and $d > 0$ such that $\tfrac{1}{T}\int_{0}^{T} g(t,u(t))\,dt > \sigma$
for each $u\in \mathcal{C}^{1}_{T}$ such that $u(t) \geq d$ for all $t\in \mathopen{[}0,T\mathclose{]}$.}
\end{itemize}

We are now in position to state and prove our first main result.

\begin{theorem}\label{th-3.1}
Assume $(H_{0})$, $(H^{\rm{I}}_{1})$, $(H^{\rm{I}}_{2})$ and $(H^{\rm{I}}_{3})$ and let
\begin{equation}\label{sigma*}
\sigma^{*} := \sup \bigl{\{}\sigma\in\mathopen{]}g_{0},+\infty\mathclose{[} \colon \text{$(H^{\rm{I}}_{3})$ is satisfied\,}\bigr{\}}.
\end{equation}
Then, there exists $s_{0}\in \mathopen{]}-\infty,\sigma^{*}\mathclose{[}$ such that:
\begin{itemize}
\item for $s_{0} < s < \sigma^{*}$, equation $(\mathscr{E}_{s})$ has at least one $T$-periodic solution;
\item for $s < s_{0}$, equation $(\mathscr{E}_{s})$ has no $T$-periodic solutions.
\end{itemize}
Moreover, if $(H^{\rm{I}}_{4})$ holds, then for
\begin{equation*}
\sigma^{**} := \sup \bigl{\{}\sigma\in\mathopen{]}g_{0},\sigma^{*}\mathclose{]} \colon \text{$(H^{\rm{I}}_{4})$ is satisfied\,}\bigr{\}}
\end{equation*}
it follows that:
\begin{itemize}
\item for $s =s_{0}$, equation $(\mathscr{E}_{s})$ has at least one $T$-periodic
solution;
\item for $s_{0} < s < \sigma^{**}$, equation $(\mathscr{E}_{s})$ has at least two $T$-periodic solutions.
\end{itemize}
\end{theorem}

\begin{proof}
We split the proof into two steps.
In the first one, we prove that there are no solutions to $(\mathscr{E}_{s})$
if the parameter $s$ is sufficiently small. Moreover, we show that the set of the parameters $s$
for which $(\mathscr{E}_{s})$ has at least one $T$-periodic solution is an interval.
In the second one, we discuss the existence and the multiplicity of solutions to $(\mathscr{E}_{s})$ in dependence of the parameter $s$.
We follow the approach in \cite{FMN-86}, by adapting also some arguments from \cite{SoZa-18}.
We consider $h_{s}(t,u):=g(t,u)-s$ to deal with
an equation of the form \eqref{eq-2.1}.

\smallskip
\noindent\textit{Step~1.} If $u$ is a $T$-periodic solution to $(\mathscr{E}_{s})$, then we have $\tfrac{1}{T}\int_{0}^{T} g(t,u(t))\,dt = s$, taking the average of the equation on $\mathopen{[}0,T\mathclose{]}$. Hence, from condition $(H^{\rm{I}}_{1})$, it follows that equation $(\mathscr{E}_{s})$ has no $T$-periodic solutions for
\begin{equation}\label{no-sol}
s < s^{\#}:= - \dfrac{1}{T}\int_{0}^{T} \gamma_{0}(t)\,dt.
\end{equation}

It is worth noting that assumption $(H_{0})$ implies that
the function $h_{s}$ satisfies $(A_{0})$. For each $s > g_{0}$, the constant function $\beta\equiv u_{0}$ is a strict upper solution to $(\mathscr{E}_{s})$. Indeed, we observe that
\begin{equation*}
(\phi(\beta'(t)))' + f(\beta(t))\beta'(t) + h_{s}(t,\beta(t))
= g(t,u_{0}) - s \leq g_{0} - s < 0
\end{equation*}
and so, by Lemma~\ref{lem-strictupper} we have the claim.

Let $\sigma_{1}$ satisfying assumption $(H^{\rm{I}}_{3})$ so that the Villari's condition at $-\infty$ with $\delta =1$ holds. Therefore, we are in position to apply Theorem~\ref{th-deg} and we obtain the existence of at least one $T$-periodic solution $u$ of $(\mathscr{E}_{s})$ for $s=\sigma_{1}$ with $u<u_{0}$.

We claim now that if $w$ is a $T$-periodic solution to $(\mathscr{E}_{s})$ for some $s=\tilde{\sigma} < \sigma_{1}$, then $(\mathscr{E}_{s})$ has a $T$-periodic solution for each $s\in \mathopen{[}\tilde{\sigma},\sigma_{1}\mathclose{]}$. Indeed, let $s\in \mathopen{]}\tilde{\sigma},\sigma_{1}\mathclose{[}$, then by applying Lemma~\ref{lem-strictupper}, we notice that $w$ is a strict upper solution to $(\mathscr{E}_{s})$,
since
\begin{equation*}
(\phi(w'(t)))' + f(w(t))w'(t) + g(t,w(t)) - s= \tilde{\sigma}-s < 0.
\end{equation*}
Moreover, as observed above, for $\sigma \in \mathopen{[}\tilde{\sigma},\sigma_{1}\mathclose{]}$, assumption $(H^{\rm{I}}_{3})$ implies again the Villari's condition at $-\infty$ with $\delta =1$. In this manner, by Theorem~\ref{th-deg} there exists at least one $T$-periodic solution $u$ of $(\mathscr{E}_{s})$ for $s=\sigma$ with $u<w$.

Recalling \eqref{no-sol}, we have deduced that the set of the parameters $s\leq\sigma_1$ for which equation $(\mathscr{E}_{s})$
has $T$-periodic solutions is an interval bounded from below (by $s^{\#}$). Let
\begin{equation*}
s_{0}:= \inf \bigl{\{} s\in\mathbb{R} \colon \text{$(\mathscr{E}_{s})$ has at least one $T$-periodic solution} \bigr{\}}.
\end{equation*}
By the arbitrary choice of $\sigma_{1}$ and the definition of $\sigma^{*}$, we conclude that there exists at least a $T$-periodic solution to $(\mathscr{E}_{s})$ for each $s\in\mathopen{]}s_{0},\sigma^{*}\mathclose{[}$.

\smallskip
\noindent\textit{Step~2.} Let $N_s$ the Nemytskii operator associated with $f(u)u'+h_{s}(t,u)$, namely $(N_{s}u)(t):= f(u(t))u'(t)+ h_{s}(t,u(t))$, $u\in X$.
Defining
\begin{equation*}
\mathcal{G}_{\lambda,s}u:=Pu+ \lambda QN_{s}u + \lambda \mathcal{K}N_{s}u, \quad u\in X, \; \lambda\in\mathopen{]}0,1\mathclose{]}.
\end{equation*}
we obtain that problem \eqref{pb-lambda} is equivalent to $u=\mathcal{G}_{\lambda,s}u$.

Let $\sigma_{1}$ satisfy assumptions $(H^{\rm{I}}_{3})$ and $(H^{\rm{I}}_{4})$. We claim that there exists a positive constant
$\Lambda =\Lambda(\sigma_{1})$ such that for each $s \leq \sigma_{1}$ any solution of
$u=\mathcal{G}_{\lambda,s}u$, with $0 < \lambda \leq 1$, satisfies $\|u\|_{\infty} < \Lambda$.

An application of Lemma~\ref{lem-2.3} ensures that
$\max u > - d$ and $\min u < d$,
for any possible $T$-periodic solution $u$ to
\begin{equation}\label{eq-3.2}
(\phi(u'(t)))' +\lambda f(u(t))u'(t) + \lambda h_{s}(t,u(t)) = 0, \quad \lambda\in\mathopen{]}0,1\mathclose{]}.
\end{equation}
Moreover, by $(H^{\rm{I}}_{1})$ and $s \leq \sigma_{1}$, it follows that
$h_s(t,u) \geq -\gamma_{0}(t) - \sigma_{1}$, for a.e.~$t\in \mathopen{[}0,T\mathclose{]}$.
Now we apply Lemma~\ref{lem-max-min} with $\gamma(t):= \gamma_{0}(t) + |\sigma_{1}|$,
and so there exists a positive constant $K_{0}=K_{0}(\sigma_{1})$
such that $\max u - \min u \leq K_{0}$,
for any possible $T$-periodic solution $u$ to \eqref{eq-3.2}.
Thus the claim follows, since by the above inequalities, we have that $\|u\|_{\infty} < \Lambda(\sigma_{1}):= K_{0}(\sigma_{1}) + d$.

Let us fix now a constant $\sigma_{2} < s^{\#}$. Let also $\rho_{g}$ be a non-negative $L^1$-function such that
$|h_s(t,u)| \leq \rho_{g}(t) + \max\{\sigma_{1},|\sigma_{2}|\}$, for a.e.~$t\in \mathopen{[}0,T\mathclose{]}$,
for all $s\in \mathopen{[}\sigma_{2},   \sigma_{1}\mathclose{]}$, for all $u\in \mathopen{[}-\Lambda(\sigma_{1}),\Lambda(\sigma_{1})\mathclose{]}$.
From Lemma~\ref{lem-max-min} there exists a constant $K_{1}=K_{1}(\sigma_{1},\sigma_{2}) > 0$
such that, for each $s\in \mathopen{[}\sigma_{2},\sigma_{1}\mathclose{]}$,
any solution of $u=\mathcal{G}_{\lambda,s}u$ with $0 < \lambda \leq 1$ satisfies
$\|u'\|_{\infty} < K_{1}$.

By considering the homotopic parameter $s\in \mathopen{[}\sigma_{2},\sigma_{1}\mathclose{]}$ and defining
\begin{equation*}
\Omega_{1}=\Omega_{1}(R_{0},R_{1}):= \bigl{\{}u\in \mathcal{C}^{1}_{T} \colon \|u\|_{\infty} < R_{0},\, \|u'\|_{\infty} < R_{1}\bigr{\}},
\end{equation*}
we obtain that
\begin{equation}\label{deg0}
\mathrm{d}_{\mathrm{LS}}(Id-\mathcal{G}_{1,\sigma_{1}},\Omega_{1},0) = \mathrm{d}_{\mathrm{LS}}(Id-\mathcal{G}_{1,\sigma_{2}},\Omega_{1},0) = 0, \quad \forall\, R_{0} \geq \Lambda(\sigma_{1}), \, \forall\, R_{1}\geq K_{1}.
\end{equation}

From the conclusions achieved in Step~1, $(\mathscr{E}_{s})$ has a $T$-periodic solution for every $s\in\mathopen{]}s_{0},\sigma^{**}\mathclose{[}\subseteq\mathopen{]}s_{0},\sigma^{*}\mathclose{[}$.

Let $\tilde{u}_{1}$ be a $T$-periodic solution to $(\mathscr{E}_{s})$ for some $s=\tilde{\sigma}_{1}\in\mathopen{]}s_{0},\sigma^{**}\mathclose{[}$.
Let us fix $s\in\mathopen{]}\tilde{\sigma}_{1},\sigma^{**}\mathclose{[}$ and claim that a second solution to $(\mathscr{E}_{s})$ exists.
Clearly, since $s>\tilde{\sigma}_{1}$, it follows that $\tilde{u}_{1}$ is a strict upper solution to $(\mathscr{E}_{s})$.

By the validity of the Villari's condition at $-\infty$ with $\delta =1$ and an application of Theorem~\ref{th-deg} we have
\begin{equation}\label{deg-1}
\mathrm{d}_{\mathrm{LS}}(Id-\mathcal{G}_{1,s},\Omega^{\rm{I}}(R_{0},w,R_{1}),0) = -1,
\end{equation}
where $R_{0} \geq \Lambda(\sigma_{1}) +1$ and $R_{1} \geq K_{1}$.

Now, from \eqref{deg0}, \eqref{deg-1} and $\Omega^{\rm{I}}(R_{0},w,R_{1}) \subseteq \Omega_{1}$, we obtain that there exists also a second solution to $(\mathscr{E}_{s})$ contained in $\Omega_{1}\setminus\overline{\Omega^{\rm{I}}(R_{0},w,R_{1})}$, via the additivity property of the topological degree.

We conclude the proof by showing that for $s=s_{0}$ there is at least one $T$-periodic solution.

Let us fix $\sigma_{1},\sigma_{2}$ with $\sigma_{2} < s_{0} < \sigma_{1} < \sigma^{**}$. Let $(s_{n})_{n}\subseteq \mathopen{]}s_{0},\sigma_{1}\mathclose{]}$ be a decreasing sequence with $s_{n}\to s_{0}$. By the above estimates, for each $n$ there exists at least one $T$-periodic solution $w_{n}$ to
\begin{equation*}
(\phi(w_{n}'(t)))' + f(w_{n}(t))w_{n}'(t) + g(t,w_{n}(t)) = s_{n}
\end{equation*}
with $\|w_{n}\|_{\infty} \leq \Lambda(\sigma_{1})$ and $\|w'_{n}\|_{\infty} \leq K_{1}(\sigma_{1},\sigma_{2})$.
Passing to the limit as $n\to \infty$ and applying Ascoli--Arzel\`{a} theorem, we achieve the existence of at least
one $T$-periodic solution to $(\mathscr{E}_{s})$ for $s=s_{0}$, concluding the proof.
\end{proof}

The following hypotheses will be assumed in the next result.
\begin{itemize}
\item[$(H^{\rm{II}}_{1})$]
\textit{There exists $\gamma_{0}\in L^{1}(\mathopen{[}0,T\mathclose{]},\mathbb{R}^{+})$ such that $g(t,u) \leq \gamma_{0}(t)$, for all $u\in\mathbb{R}$ and a.e.~$t\in\mathopen{[}0,T\mathclose{]}$.}
\item[$(H^{\rm{II}}_{2})$]
\textit{There exist $u_{0},\,g_{0}\in \mathbb{R}$ such that $g(t,u_{0})\geq g_{0}$ for a.e.~$t\in \mathopen{[}0,T\mathclose{]}$.}
\item[$(H^{\rm{II}}_{3})$]
\textit{There exist $\nu<\min\{0,g_{0}\}$ and $d > 0$ such that $\tfrac{1}{T}\int_{0}^{T} g(t,u(t))\,dt < \nu$
for each $u\in \mathcal{C}^{1}_{T}$ such that $u(t) \leq - d$ for all $t\in \mathopen{[}0,T\mathclose{]}$.}
\item[$(H^{\rm{II}}_{4})$]
\textit{There exist $\nu<\min\{0,g_{0}\}$ and $d > 0$ such that $\tfrac{1}{T}\int_{0}^{T} g(t,u(t))\,dt < \nu$
for each $u\in \mathcal{C}^{1}_{T}$ such that $u(t) \geq d$ for all $t\in \mathopen{[}0,T\mathclose{]}$.}
\end{itemize}

Our second main result is the following, which can be viewed as a ``dual'' version of Theorem~\ref{th-3.1}.

\begin{theorem}\label{th-3.2}
Assume $(H_{0})$, $(H^{\rm{II}}_{1})$, $(H^{\rm{II}}_{2})$ and $(H^{\rm{II}}_{4})$ and let
\begin{equation*}
\nu^{*} := \inf \bigl{\{}\nu\in\mathopen{]}-\infty,g_{0}\mathclose{[} \colon \text{$(H^{\rm{II}}_{4})$ is satisfied\,}\bigr{\}}.
\end{equation*}
Then, there exists $s_{0}\in \mathopen{]}\nu^{*},+\infty\mathclose{[}$ such that:
\begin{itemize}
\item for $\nu^{*}<s<s_{0}$, equation $(\mathscr{E}_{s})$ has at least one $T$-periodic solution;
\item for $s>s_{0}$, equation $(\mathscr{E}_{s})$ has no $T$-periodic solutions.
\end{itemize}
Moreover, if $(H^{\rm{II}}_{3})$ holds, then for
\begin{equation*}
\nu^{**} := \sup \bigl{\{}\nu\in\mathopen{[}\nu^{*},g_{0}\mathclose{[} \colon \text{$(H^{\rm{II}}_{3})$ is satisfied\,}\bigr{\}}
\end{equation*}
it follows that:
\begin{itemize}
\item for $s =s_{0}$, equation $(\mathscr{E}_{s})$ has at least one $T$-periodic
solution;
\item for $\nu^{**}<s<s_{0}$, equation $(\mathscr{E}_{s})$ has at least two $T$-periodic solutions.
\end{itemize}
\end{theorem}

\begin{proof}
As in the proof of Lemma~\ref{lem-max-min-2}, the change of variable $x:=-u$ transforms $(\mathscr{E}_{s})$ to
\begin{equation}\label{eq-transf}
(\tilde{\phi}(x'))'+ \tilde{f}(x)x'+ \tilde{g}(t,x) = -s,
\end{equation}
where $\tilde{\phi}(\xi) = -\phi(-\xi)$, $\tilde{f}(\xi) = f(-\xi)$, and $\tilde{g}(t,\xi) = - g(t,-\xi)$.
Then we apply Theorem~\ref{th-3.1} to the $T$-periodic problem associated with \eqref{eq-transf}.
Precisely, there exists $\tilde{s}_{0}\in \mathopen{]}-\infty,\sigma^{*}\mathclose{[}$
such that equation \eqref{eq-transf} has: no $T$-periodic solutions, for $s < \tilde{s}_{0}$;
at least one $T$-periodic solution, for $s =\tilde{s}_{0}$; at least one $T$-periodic solution,
for $\tilde{s}_{0} < s < \sigma^{*}$; at least two $T$-periodic solutions, for $\tilde{s}_{0} < s < \sigma^{**}\leq\sigma^{*}$.
Defining $s_{0}:=-\tilde{s}_{0}$ and observing that $\nu^{*}=-\sigma^{*}$ and $\nu^{**}=-\sigma^{**}$, the thesis follows.
\end{proof}

\begin{remark}\label{rem-3.1}
We stress that conditions $(H^{\rm{I}}_{2})$ and $(H^{\rm{II}}_{2})$ ensure the existence of a strict upper/lower solution to $(\mathscr{E}_{s})$,
respectively, which is given by the constant function $u_{0}$. In place of those conditions, one can assume the following.
\begin{itemize}
\item[$(H^{\bigstar}_{2})$]
\textit{There exist a function $u_{0}\in\mathfrak{D}$ and $g_{0}\in \mathbb{R}$ such that
\begin{equation*}
(\phi(u_{0}'))' + f(u_{0}) u_{0}'+ g(t,u_{0}) = g_{0}.
\end{equation*}
}
\end{itemize}
Indeed, by assuming condition $(H^{\bigstar}_{2})$, we immediately have the existence of a $T$-periodic solution $u_{0}$ to $(\mathscr{E}_{s})$ for $s=g_{0}$, which in turns is a strict upper/lower to $(\mathscr{E}_{s})$ for $s>g_{0}$ and for $s<g_{0}$,
respectively.
$\hfill\lhd$
\end{remark}

We conclude the section by proving that $(H^{\bigstar}_{2})$ holds for semi-bounded nonlinearities $g$ (see Proposition~\ref{prop-3.2}).

As a first step, recalling the definitions of the Banach spaces $Z$ and $\mathfrak{D}$, and of the projector $Q$ given in Section~\ref{section-2.2}, we introduce the following subspaces
\begin{equation*}
\tilde{Z}:=\bigl{\{}w\in Z \colon Qz=0\},
\quad \tilde{\mathfrak{D}}:=\biggl{\{}u\in \mathfrak{D} \colon \int_{0}^{T}u(t)\,dt=0\biggr{\}}.
\end{equation*}
We state the following result, which is a minor variant of \cite[Lemma~2.1]{MaMa-98}, where the operator $\tilde{\mathcal{K}}$ in our context takes the form $\tilde{\mathcal{K}}=\mathcal{K}-Q\mathcal{K}$ (with the notation introduced in Section~\ref{section-2.2}).

\begin{lemma}\label{lem-3.1}
For every $w\in\tilde{Z}$ there exists unique $u\in\tilde{\mathfrak{D}}$ such that
\begin{equation}\label{eq-mm}
(\phi(u'))'=w.
\end{equation}
Furthermore, let $\tilde{\mathcal{K}}\colon\tilde{Z}\to\tilde{\mathfrak{D}}$ be the operator which associates to $w$ the unique solution $u$ to \eqref{eq-mm}. Then,
$\tilde{\mathcal{K}}$ is continuous, maps bounded sets on bounded sets, and sends equi-integrable sets into relatively compact sets.
\end{lemma}

As a second step, for $u\in\mathfrak{D}$ let $\bar{u}:=\tfrac{1}{T}\int_{0}^{T}u(t)\,dt$. Then, we have that $u=\bar{u}+\tilde{u}$, with $\tilde{u}\in\tilde{\mathfrak{D}}$.
We deal now with the problem
\begin{equation}\label{eq-aam}
\begin{cases}
\,(\phi(\tilde{u}'))'=\lambda \biggl{(} F(t,\bar{u}+\tilde{u},\tilde{u}')- \dfrac{1}{T} \displaystyle \int_{0}^{T}F(\xi,\bar{u}+\tilde{u}(\xi),\tilde{u}'(\xi)) \,d\xi\biggr{)},\\
\,\tilde{u}\in\tilde{\mathfrak{D}}, \quad \lambda\in\mathopen{[}0,1\mathclose{]},
\end{cases}
\end{equation}
which can be equivalently written as a fixed point problem of the form
\begin{equation*}
\tilde{u}=\tilde{\mathcal{N}}(\bar{u},\tilde{u};\lambda):=\tilde{\mathcal{K}}(\lambda\mathcal{F}(\bar{u}+\tilde{u})),\quad \tilde{u}\in\tilde{\mathfrak{D}}, \, \lambda\in\mathopen{[}0,1\mathclose{]},
\end{equation*}
where $(\mathcal{F}u)(t)=F(t,u(t),u'(t))-QF(\cdot,u,u')(t)$, for $t\in\mathopen{[}0,T\mathclose{]}$.
In this setting the following result adapted from \cite{Om-88} holds (see also~\cite{AAM-78,FuPe-82,MRZ-00}).

\begin{lemma}\label{lem-3.2}
Suppose that there exists $\bar{u}_{0}\in\mathbb{R}$ such that for $\bar{u}=\bar{u}_{0}$ the set of solutions $\tilde{u}$ to \eqref{eq-aam} with $\lambda\in\mathopen{[}0,1\mathclose{]}$ is bounded. Moreover, assume that for any $M>0$ there exists $M'>0$ such that if $|\bar{u}|\leq M$ then $\|\tilde{u}\|_{\mathcal{C}^{1}}\leq M'$, where $(\bar{u},\tilde{u})$ is a solution pair to \eqref{eq-aam} for $\lambda=1$. Then, there exists a closed and connected set $\mathscr{C}\subseteq \mathbb{R}\times\mathcal{C}^{1}_{T}$ of solutions pairs $(\bar{u},\tilde{u})$ to \eqref{eq-aam} for $\lambda=1$ such that $\{\bar{u}\in\mathbb{R}\colon(\bar{u},\tilde{u})\in\mathscr{C}\}=\mathbb{R}$.
\end{lemma}

As a third step, we present an application of Lemma~\ref{lem-3.2} for problem
\begin{equation}\label{eq-AAM-lambda}
\begin{cases}
\,(\phi(\tilde{u}'))'+f(\bar{u}+\tilde{u})\tilde{u}'+\lambda g(t,\bar{u}+\tilde{u})-\dfrac{\lambda}{T}\displaystyle \int_{0}^{T}g(\xi,\bar{u}+\tilde{u}(\xi)) \, d\xi =0, \\
\,\tilde{u}\in\tilde{\mathfrak{D}}, \quad \lambda\in\mathopen{[}0,1\mathclose{]}.
\end{cases}
\end{equation}

\begin{proposition}\label{prop-3.1}
Assume that there exists a function $\rho\in L^{1}(\mathopen{[}0,T\mathclose{]},\mathbb{R}^{+})$ such that $|g(t,u)|\leq\rho(t)$ for a.e.~$t\in\mathopen{[}0,T\mathclose{]}$ and for all $u\in\mathbb{R}$.
Then, the following results hold.
\begin{itemize}
\item [$(i)$] There exists $K=K(\phi,\rho)>0$ such that any solution pair $(\bar{u},\tilde{u})$ to \eqref{eq-AAM-lambda} satisfies $\|\tilde{u}'\|_{L^{1}}\leq K$ and $\|\tilde{u}\|_{\infty}\leq K$.
\item [$(ii)$] For any $M>0$ there exists $M'=M'(\phi,f,\rho)>0$ such that if $|\bar{u}|\leq M$ then $\|\tilde{u}\|_{\mathcal{C}^{1}}\leq M'$, where $(\bar{u},\tilde{u})$ is a solution pair to \eqref{eq-AAM-lambda} for $\lambda=1$.
\end{itemize}
\end{proposition}

\begin{proof}
Let $\lambda\in\mathopen{[}0,1\mathclose{]}$ and let $(\bar{u},\tilde{u})$ be a solution pair to \eqref{eq-AAM-lambda}. We proceed similarly as in the proof of Lemma~\ref{lem-max-min}. Multiplying \eqref{eq-AAM-lambda} by $\tilde{u}$ and integrating on $\mathopen{[}0,T\mathclose{]}$, we obtain
\begin{equation*}
\int_{0}^{T} (\phi(\tilde{u}'(\xi)))' \tilde{u}(\xi) \,d\xi \leq \|\rho\|_{L^{1}} \|\tilde{u}\|_{\infty}.
\end{equation*}
We notice that for every $b>0$ there exists $K_{b}>0$ such that $\phi(\xi)\xi\geq b|\xi|-K_{b}$, for every $\xi\in\mathbb{R}$. Hence, via an integration by parts, it follows that
\begin{align*}
&\int_{0}^{T} (\phi(\tilde{u}'(\xi)))' \tilde{u}(\xi) \,d\xi =\int_{0}^{T} \phi(\tilde{u}'(\xi)) \tilde{u}'(\xi) \,d\xi \\
&\geq \int_{0}^{T} \bigl{(}b|\tilde{u}'(\xi)|-K_{b}\bigr{)} \,d\xi
= b\|\tilde{u}'\|_{L^{1}}-K_{b}T.
\end{align*}
Finally, we obtain
\begin{equation*}
b\|\tilde{u}'\|_{L^{1}} \leq \|\rho\|_{L^{1}} \|\tilde{u}\|_{\infty} + K_{b}T \leq \|\rho\|_{L^{1}} \|\tilde{u}'\|_{L^{1}} + K_{b}T.
\end{equation*}
Then, taking $b>\|\rho\|_{L^{1}}$ and $K=K(\phi,\rho):=K_{b}T/(b-\|\rho\|_{L^{1}})$, we have $\|\tilde{u}'\|_{L^{1}}\leq K$ and so $\|\tilde{u}\|_{\infty}\leq K$. Hence, $(i)$ is proved.

Let $(\bar{u},\tilde{u})$ be a solution pair to \eqref{eq-AAM-lambda} for $\lambda=1$. Let $M>0$ and suppose that $|\bar{u}|\leq M$. By the assumptions on $f$ and $g$, and the above remarks, we have that $f(\bar{u}+\tilde{u})\tilde{u}'+ g(t,\bar{u}+\tilde{u})-\tfrac{1}{T}\int_{0}^{T}g(\xi,\bar{u}+\tilde{u}(\xi)) \, d\xi$ is bounded in $L^{1}$. Next, proceeding as in the last step of the proof of Lemma~\ref{lem-max-min}, we have $\|\tilde{u}'\|_{\infty}\leq K_{1}$ and so, from $(i)$, $\|\tilde{u}\|_{\mathcal{C}^{1}}\leq M' := K+K_{1}$. The proof of $(ii)$ is completed.
\end{proof}

In the next proposition we combine the results of Proposition~\ref{prop-3.1} with an argument exploited in \cite{BeMa-08}.
\begin{proposition}\label{prop-3.2}
Assume that there exists a function $\rho\in L^{1}(\mathopen{[}0,T\mathclose{]},\mathbb{R}^{+})$ such that $|g(t,u)|\leq\rho(t)$ for a.e.~$t\in\mathopen{[}0,T\mathclose{]}$ and for all $u\geq0$ (or for all $u\leq 0$,
respectively). Then, for every $d>0$ there exists $g_{0}\in\mathbb{R}$ such that equation $(\mathscr{E}_{s})$ for $s=g_{0}$ has a $T$-periodic solution $u_{0}$ with $u_{0}(t)\geq d$ for all $t\in\mathopen{[}0,T\mathclose{]}$ (or with $u_{0}(t)\leq- d$ for all $t\in\mathopen{[}0,T\mathclose{]}$,
respectively).
 \end{proposition}

\begin{proof}
Let us suppose that $|g(t,u)|\leq\rho(t)$ for a.e.~$t\in\mathopen{[}0,T\mathclose{]}$ and for all $u\geq0$. Let us define the nonlinearity $\hat{g}\colon\mathopen{[}0,T\mathclose{]}\times\mathbb{R}\to\mathbb{R}$ as follows
\begin{equation*}
\hat{g}(t,u) :=
\begin{cases}
\, g(t,0), &\text{if $u\leq 0$,}\\
\, g(t,u), &\text{if $u\geq 0$.}
\end{cases}
\end{equation*}
We notice that $|\hat{g}(t,u)|\leq\rho(t)$ for a.e.~$t\in\mathopen{[}0,T\mathclose{]}$ and for all $u\in\mathbb{R}$.
Let us consider the equation
\begin{equation*}
(\phi(u'))'+f(u)u'+\hat{g}(t,u)=s.
\leqno{(\hat{\mathscr{E}}_{s})}
\end{equation*}
As a direct application of Lemma~\ref{lem-3.2} and Proposition~\ref{prop-3.1}, we obtain that there exists a continuum $\mathscr{C}\subseteq \mathbb{R}\times \tilde{\mathfrak{D}}$ of solution pairs $(\bar{u},\tilde{u})$ to \eqref{eq-AAM-lambda} for $\lambda=1$ such that $\{\bar{u}\in\mathbb{R}\colon(\bar{u},\tilde{u})\in\mathscr{C}\}=\mathbb{R}$. As a consequence, for every $\bar{u}\in\mathbb{R}$ there exists a solution $\tilde{u}\in\tilde{\mathfrak{D}}$ to \eqref{eq-AAM-lambda} for $\lambda=1$ and satisfying conditions $(i)$ and $(ii)$.
Let $\bar{u}_{0}\geq d+K$ and let $\tilde{u}_{0}$ be the corresponding solution to \eqref{eq-AAM-lambda} for $\lambda=1$. Let us define $u_{0}:=\bar{u}_{0}+\tilde{u}_{0}$. We notice that $u_{0}$ is a $T$-periodic solution to $(\hat{\mathscr{E}}_{s})$ for $s=g_{0}:=\tfrac{1}{T}\int_{0}^{T} \hat{g}(t,u_0(t))\,dt$. Moreover, $u_{0}(t)\geq d+K-\|\tilde{u}_{0}\|_{\infty}>d$, for all $t\in\mathopen{[}0,T\mathclose{]}$. Then $u_{0}$ is a $T$-periodic solution to $(\mathscr{E}_{s})$ for $s=g_{0}$ with $u_{0}(t)\geq d$, for all $t\in\mathopen{[}0,T\mathclose{]}$.

If we assume that $|g(t,u)|\leq\rho(t)$ for a.e.~$t\in\mathopen{[}0,T\mathclose{]}$ and for all $u\leq0$, one can proceed in a similar manner.
The theorem is thus proved.
\end{proof}

\section{Applications}\label{section-4}

In this final section, we present two consequences of the theorems illustrated in Section~\ref{section-3}. More precisely, first we show some results in the framework of $T$-periodic forced Li\'{e}nard-type equations for which theorems illustrated in the introduction are straightforward corollaries. Secondly, we analyse Neumann problems in the framework of radially symmetric solutions to partial differential equations.

\subsection{Weighted periodic problems}\label{section-4.1}

We deal with the $T$-periodic forced Li\'{e}nard-type equation
\begin{equation*}
(\phi(u'))'+f(u)u'+a(t)q(u)=s+e(t),
\leqno{(\mathscr{W\!E}_{s})}
\end{equation*}
where $s\in\mathbb{R}$ is a parameter, $\phi\colon\mathbb{R}\to\mathbb{R}$ is an increasing homeomorphism such that $\phi(0)=0$, the functions $f,q\colon\mathbb{R}\to\mathbb{R}$
are continuous. We also assume $a\in L^{\infty}(0,T)$ and $e\in L^{1}(0,T)$.
Moreover, we suppose $a(t)\geq 0$ for a.e.~$t\in\mathopen{[}0,T\mathclose{]}$ with $\bar{a}:=\tfrac{1}{T}\int_{0}^{T}a(t)\,dt>0$.
When the limits of the continuous function $q$ at $\pm\infty$ exist, we set
\begin{equation}\label{cond-q}
\lim_{u\to-\infty}q(u)=\omega_{-},
\quad
\lim_{u\to+\infty}q(u)=\omega_{+}.
\end{equation}

In the sequel we apply the general results achieved in Section~\ref{section-3} to the broadest class of nonlinear terms $q$. In order to do this, we observe that it is not restrictive to assume that $\bar{e}:=\tfrac{1}{T}\int_{0}^{T}e(t)\,dt=0$,
and, moreover, that $\min\{\omega_{-}, \omega_{+}\}>0$, if $q$ is bounded from below,
or that $\max\{\omega_{-}, \omega_{+}\}<0$, if $q$ is bounded from above.
Indeed, if necessary, one can include in the forcing term $e(t)$ the function $a(t)(-\inf q+\varepsilon)$ (or $a(t)(-\sup q-\varepsilon)$,
respectively) for some $\varepsilon>0$, and, next, add the mean value $\bar{e}$ in the parameter $s$.

We are now in position to present some corollaries of Theorem~\ref{th-3.1}, Theorem~\ref{th-3.2} and their variants. In more detail, we are interested in applications which always involve (AP) or (BM) alternatives, where the existence of at least two $T$-periodic solutions to $(\mathscr{W\!E}_{s})$ is considered. Beside these results, we warn that even partial alternatives, concerning only the existence of at least one or non-existence of $T$-periodic solutions, could be performed within our framework.

For nonlinearities $q$ bounded from below, the following result holds true.

\begin{theorem}\label{th-4.1}
Assume that there exists a number $u_{0}\in\mathbb{R}$ such that
\begin{equation*}
0 \leq q(u_{0})<\min\{\omega_{-},\omega_{+}\}
\end{equation*}
and
\begin{equation}\label{cond-ex1}
\bar{a} \min\{\omega_{-},\omega_{+}\} > \|a\|_{\infty}q(u_{0})+\|e^{-}\|_{\infty}.
\end{equation}
Then, there exists $s_{0}\in \mathopen{]}-\infty,\bar{a}\omega_{-}\mathclose{[}$ such that:
\begin{itemize}
\item for $s<s_{0}$, equation $(\mathscr{W\!E}_{s})$ has no $T$-periodic solutions;
\item for $s =s_{0}$, equation $(\mathscr{W\!E}_{s})$ has at least one $T$-periodic solution;
\item for $s_{0}<s<\bar{a}\omega_{-}$, equation $(\mathscr{W\!E}_{s})$ has at least one $T$-periodic solution;
\item for $s_{0}<s<\bar{a}\min\{\omega_{-},\omega_{+}\}$, equation $(\mathscr{W\!E}_{s})$ has at least two $T$-periodic solutions.
\end{itemize}
\end{theorem}

\begin{proof}
We apply Theorem~\ref{th-3.1} for $g(t,u) := a(t) q(u) - e(t)$.
We notice that, since $q$ is continuous and $a$ is an $L^{\infty}$-function, condition $(H_{0})$ is satisfied.
Moreover, since $q(u_{0})<\min\{\omega_{-},\omega_{+}\}$, we obtain that $q_{0}:=\min q\in\mathbb{R}$ is well defined.
Then, defining $\gamma_{0}$ as the negative part of $a(t) q_{0} - e(t)$, condition $(H^{\rm{I}}_{1})$ holds.
Furthermore, for $g_{0}= \|a\|_{\infty} q(u_{0})+ \|e^{-}\|_{\infty}$, $(H^{\rm{I}}_{2})$ is satisfied too.

Lastly, we need that both Villari's type conditions $(H^{\rm{I}}_{3})$ and $(H^{\rm{I}}_{4})$ are satisfied as well.
Let $\sigma \in\mathopen{]}g_{0}, \bar{a} \omega_{-}\mathclose{[}$ and notice that the interval is well defined and non-empty, since by hypothesis we have $g_{0} <\bar{a} \omega_{-}$.
From \eqref{cond-q} it follows that there exists $\kappa_{\sigma}>0$ such that $\bar{a}q(\xi) > \sigma$, for all $\xi\leq-\kappa_{\sigma}$.
Let $d=\kappa_{\sigma}$ and let $u\in\mathcal{C}^{1}_{T}$ be such that $u(t) \leq -d$ for every $t\in\mathopen{[}0,T\mathclose{]}$. From the generalized mean value theorem, there exists $\tilde{t}\in\mathopen{[}0,T\mathclose{]}$ such that the following holds
\begin{equation*}
\dfrac{1}{T} \int_{0}^{T} g(t,u(t)) \,dt = \dfrac{1}{T} \int_{0}^{T} a(t)q(u(t)) \,dt = \bar{a} q(u(\tilde{t})) >\sigma
\end{equation*}
and so $(H^{\rm{I}}_{3})$ is satisfied.
By the arbitrary choice of $\sigma$ we have that $(H^{\rm{I}}_{3})$ is satisfied for every $\sigma \in\mathopen{]}g_{0}, \bar{a} \omega_{-}\mathclose{[}$.
Recalling the definition of $\sigma^{*}$ in \eqref{sigma*}, we claim that $\sigma^{*}=\bar{a} \omega_{-}$. Indeed, if $\omega_{-}=+\infty$, then the claim is straightforward verified. If $\omega_{-}<+\infty$, the claim is reached by noticing that $(H^{\rm{I}}_{3})$ is not true for $\sigma>\bar{a}\omega_{-}$. Analogously, one can prove that $(H^{\rm{I}}_{4})$ holds for every $\sigma \in\mathopen{]}g_{0}, \sigma^{**}\mathclose{[}$, with $\sigma^{**}=\bar{a}\min\{\omega_{-},\omega_{+}\}$. The thesis follows.
\end{proof}

Applying the variant of Theorem~\ref{th-3.1} introduced in Remark~\ref{rem-3.1} we have the following result.

\begin{theorem}\label{th-4.2}
Assume that there exists $D>0$ such that
\begin{equation}\label{cond-ex2}
q(u)<\omega_{-},\; \text{ for $u\leq-D$,} \quad\text{ and }\quad q(u)<\omega_{+},\; \text{ for $u\geq D$.}
\end{equation}
Moreover, suppose that $\min\{\omega_{-},\omega_{+}\}<+\infty$. Then, the conclusion of Theorem~\ref{th-4.1} holds.
\end{theorem}

\begin{proof}
Let us suppose that $\min\{\omega_{-},\omega_{+}\}=\omega_{-}<+\infty$.
We apply Theorem~\ref{th-3.1} for $g(t,u) := a(t) q(u) - e(t)$. The conditions $(H_{0})$ and $(H^{\rm{I}}_{1})$ are verified as in the proof of Theorem~\ref{th-4.1}. Furthermore, an application of Proposition~\ref{prop-3.2} ensures that condition $(H^{\bigstar}_{2})$ is satisfied for some $g_{0}\in\mathbb{R}$ and $u_{0}\in\mathcal{C}^{1}_{T}$.
Indeed, since $\omega_{-}<+\infty$, it follows that there exist $g_{0}\in\mathbb{R}$ and $u_{0}\in\mathcal{C}^{1}_{T}$ such that $u_{0}(t) \leq -D$ for every $t\in\mathopen{[}0,T\mathclose{]}$, with $D>0$ as in the statement.
Then, on the light of Remark~\ref{rem-3.1}, one has only to verify the Villari's type conditions $(H^{\rm{I}}_{3})$ and $(H^{\rm{I}}_{4})$.
In order to do this, we first observe that
\begin{equation*}
g_{0}=\dfrac{1}{T} \int_{0}^{T} a(t)q(u_{0}(t)) \,dt< \bar{a}\omega_{-}\leq \bar{a}\omega_{+}.
\end{equation*}
Then, as in the proof of Theorem~\ref{th-4.1}, we have that $(H^{\rm{I}}_{3})$ is satisfied for every $\sigma \in\mathopen{]}g_{0}, \bar{a} \omega_{-}\mathclose{[}$ and $(H^{\rm{I}}_{4})$ is verified too for every $\sigma \in\mathopen{]}g_{0}, \bar{a}\min\{\omega_{-},\omega_{+}\}\mathclose{[}$.

On the other hand, if we suppose that $\min\{\omega_{-},\omega_{+}\}=\omega_{+}<+\infty$, we achieve the thesis in a similar way.
\end{proof}

Analogously, the following results for nonlinearities $q$ bounded from above can be obtained as an application of Theorem~\ref{th-3.2} (cf.~also Remark~\ref{rem-3.1}).

\begin{theorem}\label{th-4.3}
Assume that there exists a number $u_{0}\in\mathbb{R}$ such that
\begin{equation*}
0 \geq q(u_{0})>\max\{\omega_{+},\omega_{-}\}
\end{equation*}
and
\begin{equation}\label{cond-ex3}
\bar{a}\max\{\omega_{-},\omega_{+}\} < \|a\|_{\infty}q(u_{0})-\|e^{+}\|_{\infty}.
\end{equation}
Then, there exists $s_{0}\in \mathopen{]}\bar{a}\omega_{-},+\infty\mathclose{[}$ such that:
\begin{itemize}
\item for $s>s_{0}$, equation $(\mathscr{W\!E}_{s})$ has no $T$-periodic solutions;
\item for $s =s_{0}$, equation $(\mathscr{W\!E}_{s})$ has at least one $T$-periodic solution;
\item for $\bar{a}\omega_{-}<s<s_{0}$, equation $(\mathscr{W\!E}_{s})$ has at least one $T$-periodic solution;
\item for $\bar{a}\max\{\omega_{-},\omega_{+}\}<s<s_{0}$, equation $(\mathscr{W\!E}_{s})$ has at least two $T$-periodic solutions.
\end{itemize}
\end{theorem}

\begin{theorem}\label{th-4.4}
Assume that there exists $D>0$ such that
\begin{equation}\label{cond-ex4}
q(u)>\omega_{-},\; \text{ for $u\leq-D$,} \quad\text{ and }\quad q(u)>\omega_{+},\; \text{ for $u\geq D$.}
\end{equation}
Moreover, suppose that $\max\{\omega_{-},\omega_{+}\}>-\infty$.
Then, the conclusion of Theorem~\ref{th-4.3} holds.
\end{theorem}

We conclude the discussion concerning the $T$-periodic forced Li\'{e}nard-type equation $(\mathscr{W\!E}_{s})$ by presenting some examples. In this manner, we highlight the potentiality of the results proposed in this paper that, acting in a unified framework, lead to a generalization of some classical theorems. Furthermore, our approach allows us to treat more general situations by considering several types of nonlinearities (cf.~Table~\ref{tab-1}).

\begin{table}
\begin{center}
{\captionof{table}{Illustration of the graphs of the nonlinearities $q$ considered in the examples presented in Section~\ref{section-4.1}.} \label{tab-1}}
\begin{tabular}{cc}
\toprule
\begin{tikzpicture}[scale=0.6]
\draw[->] (-4,0) -- (4,0) node[right] {$u$};
\draw[->] (0,-2) -- (0,4) node[above] {$q(u)$};
\draw[very thick,black,domain=-3.8:3.7,samples=1000] plot (\x,{abs(\x)});
\end{tikzpicture}   &
\begin{tikzpicture}[scale=0.6]
\draw[->] (-4,0) -- (4,0) node[right] {$u$};
\draw[->] (0,-2) -- (0,4) node[above] {$q(u)$};
\draw[very thick,black,domain=-3.8:3.7,samples=1000] plot (\x,{3*exp(-(\x)^2)});
\end{tikzpicture}\\
Example~\ref{ex-4.1}& Example~\ref{ex-4.2}\\
&\\ \midrule    &\\
\begin{tikzpicture}[scale=0.6]
\draw[->] (-4,0) -- (4,0) node[right] {$u$};
\draw[->] (0,-2) -- (0,4) node[above] {$q(u)$};
\draw[dashed] (-3.8,1) -- (0,1);
\draw[very thick,black,domain=-1.9:1.5,samples=1000] plot (2*\x,{(exp((\x))+1)*((\x)^2)/((\x)^2+1))});
\end{tikzpicture}   &
\begin{tikzpicture}[scale=0.6]
\draw[->] (-4,0) -- (4,0) node[right] {$u$};
\draw[->] (0,-2) -- (0,4) node[above] {$q(u)$};
\draw[very thick,black,domain=-3.8:3.7,samples=1000] plot (\x,{3*(\x*exp(-(\x)^2))});
\end{tikzpicture}\\
Example~\ref{ex-4.3}& Example~\ref{ex-4.4}\\
&\\ \midrule    &\\
\begin{tikzpicture}[scale=0.6]
\draw[->] (-4,0) -- (4,0) node[right] {$u$};
\draw[->] (0,-4) -- (0,4) node[above] {$q(u)$};
\draw[dashed] (-3.8,-1) -- (3.7,-1);
\draw[very thick,black,domain=-3.8:3.7,samples=1000] plot (\x,{((exp(-(\x)^2)+2)*((\x)^6-(\x)^4-(\x)^2)+1)/((\x)^6+1)+5*exp(-(\x)^2)-3});
\end{tikzpicture}   &
\begin{tikzpicture}[scale=0.6]
\draw[->] (-4,0) -- (4,0) node[right] {$u$};
\draw[->] (0,-4) -- (0,4) node[above] {$q(u)$};
\draw[dashed] (-3.8,-pi/2) -- (0,-pi/2);
\draw[dashed] (0,pi/2) -- (3.7,pi/2);
\draw[very thick,gray,domain=-1.55:1.55,samples=1000] plot (2*\x,{(\x)^3});
\draw[very thick,black,domain=-7.6:7.4,samples=1000] plot (0.5*\x,{rad(atan(\x))});
\end{tikzpicture}\\
Example~\ref{ex-4.5}& Example~\ref{ex-4.6}\\
&\\
\bottomrule
\end{tabular}
\end{center}
\end{table}

\begin{example}\label{ex-4.1}
Let us consider $q\colon\mathbb{R}\to\mathbb{R}$ defined as $q(u) := |u|$.
We notice that $\omega^{\pm}=+\infty$. This nonlinearity is of type~I and characterizes the classical Ambrosetti--Prodi periodic problem.
An application of Theorem~\ref{th-4.1} ensures the existence of $s_{0}\in\mathbb{R}$ such that $(\mathscr{W\!E}_{s})$ has zero, at least one or at least two $T$-periodic solutions according to $s<s_{0}$, $s=s_{0}$ or $s>s_{0}$.
\end{example}

\begin{example}\label{ex-4.2}
Let us consider $q\colon\mathbb{R}\to\mathbb{R}$ defined as $q(u) := e^{-u^{2}}$.
We notice that $q$ is a Gaussian function with $\omega_{\pm}=0$, so that is of type~II, and corresponds to the classical problem by Ward. An application of Theorem~\ref{th-4.4} ensures the existence of $s_{0}\in\mathbb{R}$ such that
$(\mathscr{W\!E}_{s})$ has zero, at least one or at least two $T$-periodic solutions according to $s>s_{0}$, $s=s_{0}$ or $0<s<s_{0}$. Moreover, by an integration on a period, one can observe that for $s\leq 0$ equation $(\mathscr{W\!E}_{s})$ has no $T$-periodic solutions.

We notice that, if we consider $q(u) := e^{-u^{2}}+\kappa$ with $\kappa\not=0$, then the (BM) alternative holds without assuming a uniform condition in the limits. More precisely, there are zero, at least one or at least two $T$-periodic solutions according to $s>s_{0}$, $s=s_{0}$ or $\bar{a}\kappa<s<s_{0}$. In the same context, one could be also driven to apply Theorem~\ref{th-4.3}. However, this an example of the difference between these results. Indeed, if for example $\kappa=-1$, then Theorem~\ref{th-4.3} ensures the existence of $s_{0}\in\mathopen{]}-\bar{a},+\infty\mathclose{[}$ such that $(\mathscr{W\!E}_{s})$ satisfies the above alternative under the additional hyphotesis $\bar{a}>\|e^{+}\|_{\infty}$ (cf.~condition~\eqref{cond-ex3}).
\end{example}

\begin{example}\label{ex-4.3}
Let us consider $q\colon\mathbb{R}\to\mathbb{R}$ defined as
\begin{equation*}
q(u) := \dfrac{(e^{u}+1)u^{2}}{u^{2}+1},
\end{equation*}
where $\omega_{-}=1$ and $\omega_{+}=+\infty$. An application of Theorem~\ref{th-4.2} guarantees the existence of $s_{0}\in\mathbb{R}$ such that $(\mathscr{W\!E}_{s})$ has zero, at least one or at least two $T$-periodic solutions according to $s<s_{0}$, $s=s_{0}$ or $s_{0}\leq s<\bar{a}\omega_{-}$. On the other hand, since $\omega_{-}\in\mathopen{]}0,+\infty\mathclose{]}$, whether the terms $a,e$ satisfy the additional condition \eqref{cond-ex1}, the same alternative can be proved via Theorem~\ref{th-4.1}.
\end{example}

\begin{example}\label{ex-4.4}
Let us consider $q\colon\mathbb{R}\to\mathbb{R}$ defined as
$q(u) := u\,e^{-u^{2}}$,
where $\omega_{\pm}=0$. We notice that either condition \eqref{cond-ex2} or condition \eqref{cond-ex4} are not satisfied, so that, both Theorem~\ref{th-4.2} and Theorem~\ref{th-4.4} cannot be applied in such a framework. However, one can recover an (AP) alternative through Theorem~\ref{th-4.1} for the auxiliary equation
$(\phi(u'))'+f(u)u'+a(t)q_{1}(u)=\ell+e_{1}(t)$,
where $q_{1}(u)=q(u)-\min q$, $e_{1}(t)=e(t)-\bar{e}-(a(t)-\bar{a})\min q$, and $\ell=s+\bar{e}-\bar{a}\min q$.
In this way $a,e_{1}$ satisfy \eqref{cond-ex1}. On the other hand, in a similar manner, one can recover a (BM) alternative through Theorem~\ref{th-4.3}.
\end{example}

\begin{example}\label{ex-4.5}
Let us consider $q\colon\mathbb{R}\to\mathbb{R}$ defined as
\begin{equation*}
q(u) := \dfrac{(e^{-u^{2}}+2)(u^{6}-u^{4}-u^{2}+1)}{u^{6}+1}+5e^{-u^{2}}-3,
\end{equation*}
where $\omega_{\pm}=-1$. We notice that condition \eqref{cond-ex2} is satisfied, so that from Theorem~\ref{th-4.2} we can prove that there exists $s_{0}<-\bar{a}$ such that
$(\mathscr{W\!E}_{s})$ has zero, at least one or at least two $T$-periodic solutions according to $s<s_{0}$, $s=s_{0}$ or $s_{0}<s<-\bar{a}$. On the other hand, if condition \eqref{cond-ex3} is satisfied too, then an application of Theorem~\ref{th-4.3} gives the existence of $s_{1}>-\bar{a}$ such that $(\mathscr{W\!E}_{s})$ has zero, at least one or at least two $T$-periodic solutions according to $s>s_{1}$, $s=s_{1}$ or $-\bar{a}<s<s_{1}$. In this manner, a combination of the classical (AP) and (BM) alternatives hold simultaneously, leading to an interesting phenomenon of $0$~$1$~$2$~-~$2$~$1$~$0$ solutions to the periodic problem associated with $(\mathscr{W\!E}_{s})$.
\end{example}

\begin{example}\label{ex-4.6}
Let us consider $q\colon\mathbb{R}\to\mathbb{R}$ defined as
$q(u) := u^{2 n+1}$, with $n\in\mathbb{N}$,
or $q(u) := \mathrm{arctan}(u)$.
In these cases, $q$ is a non-decreasing function, and consequently, the above theorems do not apply. However, when the nonlinearity is bounded, one could achieve the existence of at least a $T$-periodic solution for some ranges of the parameter $s$ (cf.~\cite{MRZ-00,Om-88}).
\end{example}

\subsection{Radial Neumann problem on annular domains}\label{section-4.2}

Let us consider the open annular domain
\begin{equation*}
\mathfrak{A} := \bigl{\{} x\in\mathbb{R}^{N} \colon R_{i} < |x| < R_{e} \bigr{\}},\quad \text{with } \, 0<R_{i}<R_{e},
\end{equation*}
where $|\cdot|$ denotes the usual Euclidean norm in $\mathbb{R}^{N}$, with $N\geq2$.
In the present section we study non-existence, existence and multiplicity of (classical) radially symmetric solutions to the parameter-dependent Neumann problem
\begin{equation}\label{pb-PDE}
\begin{cases}
\, \nabla\cdot(A(|\nabla u|)\nabla u) + G(|x|,u) = s & \text{in $\mathfrak{A}$,} \\\vspace*{2pt}
\, \dfrac{\partial u}{\partial \nu} = 0 & \text{on $\partial\mathfrak{A}$,}
\end{cases}
\end{equation}
where $A\colon\mathopen{]}0,+\infty\mathclose{[}\to\mathopen{]}0,+\infty\mathclose{[}$ is a continuous function such that the map $\phi(\xi)=A(|\xi|)\xi$ for $\xi\not=0$ and $\phi(0)=0$ is a homeomorphism on the real line, $G\colon\mathopen{[}R_{i},R_{e}\mathclose{]}\times\mathbb{R}\to\mathbb{R}$ is a continuous map, and $s\in\mathbb{R}$. In this manner, we pursue the study started in \cite{So-18,SoZa-18} for Neumann problems with local coercive nonlinearities.

When dealing with radially symmetric solutions to \eqref{pb-PDE}, one is led to define $r=|x|$, $v(r)=v(|x|)=u(x)$, and so to study the problem
\begin{equation}\label{pb-ODE}
\begin{cases}
\, (r^{N-1}A(|v'|)v')' + r^{N-1} G(r,v) = r^{N-1} s, \\
\, v'(R_{i}) = v'(R_{e}) = 0.
\end{cases}
\end{equation}
We notice that the map $\xi \mapsto A(|\xi|)\xi$ is an increasing homeomorphism.
Hence, looking at solutions to \eqref{pb-ODE}, we now present our result in the framework of a more general problem. Namely, we deal with a Neumann problem of the form
\begin{equation}\label{pb-N}
\begin{cases}
\, (\zeta(t) \phi(u'))' + g(t,u(t)) = p(t)s, \\
\, u'(a) = u'(b) = 0,
\end{cases}
\end{equation}
where $a<b$, $\zeta,p\colon\mathopen{[}a,b\mathclose{]}\to\mathopen{]}0,+\infty\mathclose{[}$ are continuos functions, $\phi\colon\mathbb{R}\to\mathbb{R}$ is an increasing homeomorphism such that $\phi(0)=0$, $g\colon\mathopen{[}a,b\mathclose{]}\times\mathbb{R}\to\mathbb{R}$ satisfies Carath\'eodory conditions, and $s\in\mathbb{R}$.

First of all, let
$
X := \bigl{\{} u\in\mathcal{C}^{1}(\mathopen{[}a,b\mathclose{]}) \colon u'(a)=u'(b)=0 \bigr{\}}
$
be the Banach space endowed with the norm $\|u\|_{X}:=\|u\|_{\infty}+\|u'\|_{\infty}$.
Next we define the completely continuous operator $\mathcal{G}\colon X\to X$ as
\begin{equation*}
(\mathcal{G} u)(t) := u(a) - \dfrac{1}{b-a} \int_{a}^{b} h(t,u(t)) \,dt
+ \int_{a}^{t} \phi^{-1} \biggl{(} - \dfrac{1}{\zeta(t)} \int_{a}^{s} h(\xi,u(\xi)) \,d\xi \biggr{)} \,ds,
\end{equation*}
where $h(t,u):=g(t,u)-p(t)s$.
One can easily verify that $u$ is a solution to problem \eqref{pb-N} if and only if $u$ is a fixed point of $\mathcal{G}$ (cf.~\cite[Section~2]{GHMaZa-93}).

In this setting, up to minimal changes in the discussion in Section~2 and Section~3, all the results presented therein hold for problem \eqref{pb-N}, too.

Taking into account that
\begin{equation*}
\int_{\Omega} G(|x|,u(x))\,dx = \omega_{N-1}\int_{R_{i}}^{R_{e}} r^{N-1}G(|r|,u(r))\,dr,
\end{equation*}
where $\omega_{N-1}$ is the area of the unit sphere in $\mathbb{R}^{N}$ (cf.~\cite[Section~2.7]{Fo-84}), we obtain the following theorem for radially symmetric solutions to the Neumann problem \eqref{pb-PDE}.

\begin{theorem}\label{th-4.Neu}
Assume that
\begin{itemize}
\item[$(G_{1})$] {there exists $C_0>0$ such that $G(|x|,u) \geq - C_{0}$, for all $u\in\mathbb{R}$ and $x\in\Omega$;}
\item[$(G_{2})$] {there exists $u_{0},g_{0}\in \mathbb{R}$ such that $G(|x|,u_{0}) \leq g_{0}$, for all $x\in\Omega$;}
\item[$(G_{3})$]{for each $\sigma$ there exists $d_{\sigma} > 0$ such that $\tfrac{1}{|\Omega|}\int_{\Omega} G(|x|,u(x))\,dx > \sigma$
for each radially symmetric $u\in \mathcal{C}^{0}(\overline{\Omega})\cap\mathcal{C}^{1}(\Omega)$ with $u(x) \leq -d_{\sigma}$ for all $x\in \Omega$;}
\item[$(G_{4})$]{for each $\sigma$ there exists $d_{\sigma} > 0$ such that $\tfrac{1}{|\Omega|}\int_{\Omega} G(|x|,u(x))\,dx > \sigma$
for each radially symmetric $u\in \mathcal{C}^{0}(\overline{\Omega})\cap\mathcal{C}^{1}(\Omega)$ with $u(x) \geq d_{\sigma}$ for all $x\in \Omega$.}
\end{itemize}
Then, there exists $s_{0}\in \mathbb{R}$ such that
\begin{itemize}
\item for $s < s_{0}$, problem \eqref{pb-PDE} has no radially symmetric solutions;
\item for $s =s_{0}$, problem \eqref{pb-PDE} has at least one radially symmetric solution;
\item for $s > s_{0}$, problem \eqref{pb-PDE} has at least two radially symmetric solutions.
\end{itemize}
\end{theorem}

A direct application of Theorem~\ref{th-4.Neu} is the following one.

\begin{corollary}\label{cor-4.Neu}
Let $a\in \mathcal{C}(\mathopen{[}R_{i},R_{e}\mathclose{]},\mathbb{R}^{+})$ be such that $a(\xi_{0})>0$ for some $\xi_{0}\in\mathopen{[}R_{i},R_{e}\mathclose{]}$. Let $q\in\mathcal{C}(\mathbb{R})$ be such that $\lim_{|u|\to+\infty}q(u)=+\infty$.
Let
$G(|x|,u):= a(|x|)q(u)$.
Then, the conclusion of Theorem~\ref{th-4.Neu} holds.
\end{corollary}

\bibliographystyle{elsart-num-sort}
\bibliography{FSZ_biblio}

\end{document}